\documentclass[11pt]{article}

\usepackage{amsfonts}
\usepackage{amscd}
\usepackage{amssymb}
\usepackage{amsthm}
\usepackage{amsmath}
\usepackage{stmaryrd}
\usepackage{graphicx}
\usepackage{color}
\usepackage{verbatim}
\usepackage{mathrsfs}

\input prepictex
\input pictex
\input postpictex

 \theoremstyle{plain}
\newtheorem{thm}{Theorem}[section]
\newtheorem{lemma}[thm]{Lemma}
\newtheorem{prop}[thm]{Proposition}
\newtheorem{cor}[thm]{Corollary}

\theoremstyle{definition}

\newtheorem{remark}[thm]{Remark}
\newtheorem{example}[thm]{Example}
\numberwithin{equation}{section}

\newcommand{\cev}[1]{\reflectbox{\ensuremath{\vec{\reflectbox{\ensuremath{#1}}}}}}

\def\cA{\mathcal{A}}

\def\cC{\mathcal{C}}

\def\cF{\mathcal{F}}
\def\cG{\mathcal{G}}

\def \cP{\mathcal{P}}

\def\CC{\mathbb{C}}

\def\ZZ{\mathbb{Z}}

\def\Card{\mathrm{Card}}

\def\proj{\mathrm{proj}}
\def\dist{\mathrm{dist}}
\def\e{\mathrm{end}}

\def\scA{\mathscr{A}}
\def\scH{\mathscr{H}}

\makeatletter
\renewcommand{\@makefnmark}{\mbox{\textsuperscript{}}}
\makeatother

\title{Distance regularity in buildings and structure constants in Hecke algebras}
\author{Peter Abramenko
\and
James Parkinson\footnote{Research supported under the Australian Research Council (ARC) discovery grant DP110103205.}
\and
Hendrik Van Maldeghem}
\date{\today}

\begin{document}

\maketitle

\begin{abstract} 
In this paper we define generalised spheres in buildings using the simplicial structure and Weyl distance in the building, and we derive an explicit formula for the cardinality of these spheres. We prove a generalised notion of distance regularity in buildings, and develop a combinatorial formula for the cardinalities of intersections of generalised spheres. Motivated by the classical study of algebras associated to distance regular graphs we investigate the algebras and modules of Hecke operators arising from our generalised distance regularity, and prove isomorphisms between these algebras and more well known parabolic Hecke algebras. We conclude with applications of our main results to non-negativity of structure constants in parabolic Hecke algebras, commutativity of algebras of Hecke operators, double coset combinatorics in groups with $BN$-pairs, and random walks on the simplices of buildings. 
\end{abstract}

\section*{Introduction}

\textit{Buildings} are combinatorial/geometric objects with diverse applications in Lie theory and related fields. Spheres and intersections of spheres in buildings (primarily of spherical and affine types) have been studied due to their connections with combinatorial representation theory~\cite{Par:06,GL:05}, harmonic analysis on groups of $p$-adic type~\cite{Mac:71,Car:99,Par:06b}, incidence structures related to groups of Lie type~\cite{BCN:89}, random walk theory~\cite{Par:07}, association schemes~\cite{BCN:89}, and number theory~\cite{SS:96,Set:08}.

In this paper we define \textit{generalised spheres} in buildings of arbitrary type using the simplicial structure and Weyl distance in the building. We investigate the combinatorics of these spheres, the algebraic structures related to them, and provide applications of our results, including the non-negativity of structure constants in parabolic Hecke algebras, combinatorics of double coset decompositions in groups with $BN$-pairs, and limit theorems for random walks on buildings. 

A main motivation for this paper is to facilitate the study of Hecke algebras related to buildings beyond the spherical and affine cases. Our methodology stems from the theory of \textit{distance regular graphs}, and so we begin by briefly recalling this theory (see \cite{BCN:89}). Let $\Gamma=(V,E)$ be a locally finite graph with graph metric $d(\cdot,\cdot)$, and for each $x\in V$ and $k\in\mathbb{N}$ let
$
F_k(x)=\{y\in V\mid d(x,y)=k\}
$
be the sphere of radius $k$ centred at~$x$. The graph $\Gamma$ is \textit{distance regular} if the intersection cardinality $|F_k(x)\cap F_{\ell}(y)|$ depends only on $k,\ell$ and $d(x,y)$. In other words, there are integers $a_{k,\ell}^m$, called the \textit{intersection numbers} of the graph, such that $a_{k,\ell}^m=|F_k(x)\cap F_{\ell}(y)|$ whenever $d(x,y)=m$. In particular, the cardinality $|F_k(x)|=a_{k,k}^0$ does not depend on~$x\in V$. The infinite distance regular graphs are completely classified \cite{Iva:83}. In contrast there is no complete classification of finite distance regular graphs, although many partial results exist, for example there are precisely $13$ finite cubic distance regular graphs  (see~\cite{BBS:86}).  

A primary tool in the study of finite distance regular graphs is the \textit{Bose-Mesner algebra} of the graph. As a $\CC$-vector space, this algebra is the linear span $\scA$ of the matrices $\{A_k\mid k=0,1\ldots,d\}$ where $d$ is the diameter of $\Gamma$ and $A_k$ is the $k$-adjacency matrix of the graph (with $(x,y)$-th entry~$1$ if $d(x,y)=k$ and $0$ otherwise). The distance regular property is equivalent to the fact that the matrices $A_k$ satisfy linear relations
\begin{align}\label{eq:linearrelation}
A_kA_{\ell}=\sum_{m=0}^da_{k,\ell}^mA_m\quad\textrm{for all $k,\ell\in \{0,1,\ldots,d\}$}.
\end{align}
It follows that $\scA$ is a $\CC$-algebra, and since $a_{k,\ell}^m=a_{\ell,k}^m$ this algebra is commutative.

Let us now return to the setting of the current paper. Let $(W,S)$ be a Coxeter system, and let $X$ be a locally finite, regular building of type~$(W,S)$ with chamber set $\cC$ and Weyl distance function $\delta:\cC\times\cC\to W$. The building $X$ is a labelled simplicial complex, and so each simplex $A\in X$ has a \textit{type} $\tau(A)\subseteq S$. The \textit{cotype} of $A\in X$ is $S\backslash\tau(A)$, and for each $I\subseteq S$ we write $X_I$ for the set of all cotype~$I$ simplices of~$X$. For example, a vertex of $X$ has cotype $S\backslash\{s\}$ for some $s\in S$, and a chamber of $X$ has cotype~$\emptyset$.

The Weyl distance function may be extended to arbitrary simplices of $X$, and if $A\in X_I$ and $B\in X_J$ we have $\delta(A,B)\in R(I,J)$ where $R(I,J)$ is the transversal of minimal length $(W_I,W_J)$ double coset representatives (here $W_I$ and $W_J$ are standard parabolic subgroups of~$W$). For $A\in X_I$ and $w\in R(I,J)$ the \textit{generalised sphere} of ``radius'' $(J,w)$ centred at $A$ is by definition
$$
F(A,J,w)=\{B\in X_J\mid \delta(A,B)=w\}.
$$
If the parabolic subgroup $W_I$ is infinite then the cardinality of the sphere $F(A,J,w)$ is often infinite (even in the case of a thin building; that is, a Coxeter complex), and thus we confine our attention to the case that $W_I$ is finite, in which case $I$ is called a \textit{spherical subset} of~$S$.

The main results of this paper are as follows. 
\begin{enumerate}
\item[(1)] In Theorem~\ref{thm:sphere} we prove a new theorem which determines $|F(A,J,w)|$ in the most general situation. Only very special cases were treated before (see \cite[Proposition~2.7]{Car:99}, \cite[Theorem~5.15]{Par:06}, \cite[Proposition~3.9]{Set:08}), often with considerable effort. Our approach, making use of projections in buildings, is efficient and yields a much more general result.
\item[(2)] Having deduced the most general formula for the numbers $|F(A,J,w)|$, we turn our attention to the considerably more complicated intersection numbers $|F(A,J,u)\cap F(B,J,v)|$. In order to prove a formula for these numbers we introduce the new idea of \textit{pointed pregalleries} in the Coxeter complex, a model which is inspired by (although distinct from) the alcove walk model developed by Ram~\cite{Ram:06} and Parkinson, Ram, Schwer~\cite{PRS:09} for the study of affine Hecke algebras and affine flag varieties. Our formula for the intersection cardinalities is given in Theorem~\ref{thm:pathsv1}, and explicit computations using this formula are given in Examples~\ref{ex:F41} and~\ref{ex:E82}.
\item[(3)] In Section~\ref{sec:4} we define generalised Hecke operators on buildings and investigate the Hecke algebras $\scA(I,I)$ and modules $\scA(I,J)$ spanned by these operators. These algebras and modules are the natural analogues of the Bose-Mesner algebra of a distance regular graph. In Theorem~\ref{thm:product} we prove that our Hecke algebras $\scA(I,I)$ are isomorphic to the \textit{parabolic Hecke algebras} $\mathbf{1}_I\scH\mathbf{1}_I$ introduced in~\cite{CIK:71}, and that our modules $\scA(I,J)$ are both left and right modules for parabolic Hecke algebras. Our combinatorial interpretation of the structure constants in these parabolic Hecke algebras yields new insights, for example the positivity result given in Corollary~\ref{cor:nonneg}. 
\end{enumerate}

The interpretation of the structure constants in ``non-parabolic'' Hecke algebras in terms of buildings was made in~\cite{Par:06} (this is the case $I=J=\emptyset$ above), and association scheme theoretic interpretations of the structure constants have been developed by Zieschang~\cite{Zie:96}. In the spherical case implicit connections between Hecke algebras and buildings via finite groups with $BN$-pairs are given by Iwahori~\cite{Iwa:64}, Matsumoto~\cite{Mat:64},  Curtis, Iwahori, and Kilmoyer~\cite{CIK:71}. Moreover, Brouwer, Cohen and Neumaier \cite[Chapter~10]{BCN:89}, Brouwer and Cohen~\cite{BC:85}, and Gomi~\cite{Gom:96} compute structure constants in some spherical cases. In the present paper our point of view is as general as possible. We consider buildings of arbitrary type, with the only assumptions being local finiteness and a very mild regularity hypothesis (which is automatically satisfied in the $2$-spherical case). In this generality such buildings may not admit an ``interesting'' group action, but of course our results also cover the highly symmetric cases including locally finite affine buildings and buildings arising from Kac-Moody groups over finite fields. Also in these cases our formulae for $|F(A,J,w)|$ and $|F(A,J,u)\cap F(B,J,v)|$, and the interpretation of the parabolic Hecke algebras are new. 

Other motivations for our investigations include the applications presented in Section~\ref{sec:5}. In particular we prove a non-negativity result for the structure constants in parabolic Hecke algebras, and classify the algebras of Hecke operators on regular buildings which are commutative. We also obtain combinatorial formulae for double coset cardinalities in groups with $BN$-pairs, and outline how our algebras generalise the theory of Gelfand pairs associated to buildings. Finally we prove a local limit theorem for a random walk on the simplices of an $\tilde{A}_2$ building.

\section{Background and definitions}\label{sec:1}

\subsection{Coxeter groups and complexes}

A \textit{Coxeter system} $(W,S)$ is a group $W$ together with a
generating set~$S$ with defining relations
$$
s^2=1\quad\textrm{and}\quad (st)^{m_{st}}=1\quad\textrm{for all $s,t\in S$ with $s\neq t$},
$$
where $m_{st}=m_{ts}\in\ZZ_{\geq 2}\cup\{\infty\}$ for all $s\neq t$
(if $m_{st}=\infty$ then it is understood that there is no relation
between $s$ and $t$). We shall always assume in this paper that the
\textit{rank} of $(W,S)$, defined as $|S|$, is finite.

The \textit{length} of $w\in W$ is
$$
\ell(w)=\min\{n\geq 0\mid w=s_1\cdots s_n\textrm{ with }s_1,\ldots,s_n\in S\}.
$$
We say that a decomposition $w=s_1\cdots s_n\textrm{ with
}s_1,\ldots,s_n\in S$ is \textit{reduced} if $n = \ell(w)$. The
\textit{Bruhat order} on $W$ can be defined as follows: For $u, w \in W,$ $u < w$ if and only if a reduced decomposition of
$u$ can be obtained by deleting some of the $s_i$ in a reduced
decomposition $w=s_1\cdots s_n$ of $w$ (so that in particular
$\ell(u) < \ell(w)$).

A \textit{standard subgroup} of $W$ is by definition a subgroup of the form
$W_I=\langle \{s\mid s\in I\}\rangle$ with $I$ a subset of $S$. Such
a subset $I$ is called \textit{spherical} if $W_I$ is finite. To every
Coxeter system $(W,S)$ one associates the \textit{Coxeter complex}
$\Sigma(W,S)$, see \cite[Section 3.1]{ABr:08}. Thus $\Sigma(W,S)$ is a
thin chamber complex, its chambers are identified with the elements
of $W$, its simplices with standard cosets $wW_I$, and $wW_I$ is a
face of $uW_J$ if and only if $uW_J \subseteq wW_I$ (for $u,w \in W$ and 
$I, J \subseteq S$). The \textit{type} of a simplex $wW_I$ in
$\Sigma(W,S)$ is by definition $S \backslash I$ and its
\textit{cotype} is $I$.

\subsection{Double cosets and reduced elements}\label{sec:doublecosets}

Let $(W,S)$ be a Coxeter system. The following facts about cosets and double cosets in $W$ are
well-known, see for instance
\cite[Subsection 2.3.2]{ABr:08}:
\begin{enumerate}
\item[(a)] If $J\subseteq S$ and $w\in W$ then the coset $wW_J$ has a unique
minimal length representative $w_1$, and $\ell(w_1y) = \ell(w_1) +
\ell(y)$ for all $y \in W_J$.
\item[(b)] If $I,J\subseteq S$ and $w\in W$ then the double coset $W_IwW_J$ has
a unique minimal length representative. This representative is
called \textit{$(I,J)$-reduced}, and we let
$$
R(I,J)=\{w\in W\mid w\textrm{ is $(I,J)$-reduced}\}.
$$
Note that $R(I,J)$ indexes the decomposition of $W$ into $W_IwW_J$
double cosets.
\item[(c)] In general it is not true that $\ell(xwy) = \ell(x) + \ell(w) +
\ell(y)$ for all $x \in W_I, y \in W_J$ and $w \in R(I,J)$. In order
to address this issue, we introduce the following notation, where
$I$ and $J$ are subsets of $S$ and $w\in R(I,J)$:
\begin{align*}
W_I(J,w)&=W_I\cap wW_Jw^{-1}\\
M_I(J,w)&=\textrm{the set of all minimal length representatives of
cosets in $W_I/W_I(J,w)$}.
\end{align*}
Note that $W_I(J,w)=\{v\in W_I\mid vwW_J=wW_J\}$ is the subgroup of~$W_I$ stabilising~$wW_J$, and by \cite[Lemma 2.25]{ABr:08} we have that
\begin{align}\label{eq:stabiliser}
W_I(J,w)=W_{I\cap wJw^{-1}}.
\end{align}
Thus applying point (a) it follows that each $v\in W_I$ can be uniquely expressed in
the form 
\begin{align}\label{eq:exp1}
v = xz\text{ with $x \in M_I(J,w)$, $z \in W_I(J,w)$, and
that $\ell(v) = \ell(xz) = \ell(x) + \ell(z)$}.
\end{align}
\end{enumerate}

We now obtain an analogous result for double cosets.

\begin{lemma}\label{lem:exp} Let $I$ and $J$ be subsets of $S$ and let $w\in R(I,J)$. Then
each $v\in W_IwW_J$ can be written in exactly one way as
    $$
    v=xwz\qquad\textrm{with\quad $x\in M_{I}(J,w)$\quad and\quad $z\in W_J$}.
    $$
    Moreover, if $x\in M_I(J,w)$, $w\in R(I,J)$, and $z\in W_J$ then $\ell(xwz)=\ell(x)+\ell(w)+\ell(z)$.
\end{lemma}

\begin{proof} It follows immediately from $W_I = M_I(J,w)W_I(J,w)$
and $W_I(J,w) = W_I \cap wW_Jw^{-1}$ that $W_IwW_J = M_I(J,w)wW_J$.
If $x,x' \in M_I(J,w)$ and $y,y' \in W_J$, then $xwy = x'wy'$
implies $x^{-1}x' = wyy'^{-1}w^{-1} \in W_I(J,w)$, and so $x' = x$
is the unique minimal length representative of $xW_I(J,w)$, which
then also implies $y' = y$. Finally, assume that $\ell(xwz) <
\ell(x)+\ell(w)+\ell(z)$. Using the deletion condition, we find $x'
\in W_I$, $y' \in W_J$ with $x' < x$, $y' < y$ and $x'wy' = xwy$
(see \cite[Lemma 2.24]{ABr:08}). However, as before $xwy = x'wy'$
implies $xW_I(J,w) = x'W_I(J,w)$, contradicting the minimality of
$x$ in $xW_I(J,w)$.
\end{proof}

\subsection{Buildings}

We assume that the reader is already acquainted with the theory of buildings, and our main reference for this subject is \cite[Chapters 4 and
5]{ABr:08}. In this paper, buildings will mainly be considered as
chamber (and hence as simplicial) complexes but the associated Weyl
distance will also play an important role.

In the following, $X$ will always denote a building and $\cC$ its
set of chambers. The apartments of $X$ are all isomorphic to the
standard Coxeter complex $\Sigma(W,S)$ for a Coxeter system $(W,S)$
which is determined, up to isomorphism, by $X$. We then say that $X$
is of \textit{type} $(W,S)$ and of \textit{rank} $|S|$. If we fix an apartment $\cA$
of $X$ and an isomorphism $f : \Sigma(W,S) \rightarrow \cA$, the
\textit{type} $\tau(A)$ of a simplex $A$ of $\cA$ is by definition
the type of $f^{-1}(A)$. There is then a unique extension $\tau: X
\rightarrow \{I \mid I \subseteq S\}$ such that $\tau(A) \subseteq
\tau(B)$ whenever $A$ is a face of $B$ and $\tau(B) = S$ if and only
if $B$ is a chamber of $X$. We fix such a type function on $X$ and
define the \textit{cotype} of a simplex $A \in X$ to be $S \backslash
\tau(A)$. We set
$$
X_I=\{A\in X \mid A \textrm{ is of cotype }I \}.
$$
For $c \in \cC$ and $I \subseteq S$ the unique element of $X_I$
which is a face of $c$ is denoted by $c_I$. For $A\in X$ let 
$$
\cC(A)=\{c\in\cC\mid A \textrm{ is a face of }c\},
$$
and for any subset $M\subset X$ we write $\cC(M)=\bigcup_{A\in M}\cC(A)$. If $A \in X_I$
then $\cC(A) = \{c \in \cC \mid c_I = A \}$. \textit{Panels} of $X$
are codimension 1 simplices; an \textit{$s$-panel} is one of cotype
$\{s\}$ for $s \in S$. The building $X$ is called \textit{thick} if each panel is
contained in at least $3$ chambers.

A \textit{pregallery} in $X$ of \textit{length} $n \in \mathbb N_0$
is a sequence
$$
\gamma=(c_0,s_1,c_1,s_2,c_2,s_3,\ldots,s_n,c_n)
$$
where $s_1,\ldots,s_n\in S$ and $c_0,c_1,\ldots,c_n\in \cC$ and, for
each $j=1,\ldots,n$, either $c_{j-1} = c_j$ or $c_{j-1} \cap c_j$ is
an $s_j$-panel. The \textit{type} of the pregallery
$\gamma=(c_0,s_1,c_1,s_2,\ldots,s_n,c_n)$ is $(s_1,\ldots,s_n)$. The
type is called \textit{reduced} if $w = s_1 \cdots s_n \in W$
satisfies $\ell(w) = n$. The chamber $c_0$ is called the
\textit{start} of $\gamma$, and the chamber $c_n$ is called the
\textit{end} of $\gamma$, written $\e(\gamma)=c_n$.

A \textit{gallery} of \textit{type} $(s_1,\ldots,s_n)$ is a
pregallery $\gamma=(c_0,s_1,c_1,s_2,\ldots,s_n,c_n)$ such that
$c_{j-1} \neq c_j$ for all $j=1,\ldots,n$. Note that the type of the
gallery $\gamma$ is already determined by the sequence $(c_0, c_1
\ldots, c_n)$, although this is not true for pregalleries. For $c,d \in \cC$ the \textit{gallery distance}
$\dist(c,d)$ is the minimal length of galleries starting in $c$ and
ending in $d$.

For any two chambers $c,d \in \cC$ there exists an element $w \in W$
with the following property: There exists a gallery $\gamma$ of
reduced type $(s_1, \ldots, s_n)$ which starts in $c$ and ends in
$d$ if and only if $s_1 \cdots s_n = w$. We set $\delta(c,d) = w$
for this element $w$, thus obtaining a \textit{Weyl distance}
function $\delta : \cC \times \cC \rightarrow W$. For the basic properties of
$\delta$ we refer to \cite[Section 4.8]{ABr:08}. We mention in
passing that $\dist(c,d) = \ell(\delta(c,d))$. For $A \in X_I$ and $B
\in X_J$ (with $I,J \subseteq S$) we have
$$ \delta(\cC(A) \times \cC(B)) = W_I w W_J$$
whenever $w = \delta(c,d)$ for fixed chambers $c \in \cC(A)$ and $d
\in \cC(B)$. This allows us to extend the Weyl distance to arbitrary
simplices of $X$ by setting (for $A\in X_I$ and $B\in X_J$):
\begin{align}\label{eq:deltadist}
\delta(A,B)=\textrm{the unique $(I,J)$-reduced element of
$\delta(\cC(A)\times\cC(B))$}.
\end{align}

\subsection{Retractions and projections}\label{sec:retproj}

We now briefly recall two fundamental concepts of building theory
which we will apply later. First, for a fixed apartment $\cA$ of $X$
and a fixed chamber $a$ of $\cA$, there is the \textit{retraction
$\rho_{\cA,a}$} of $X$ onto $\cA$ centred at $a$. It is the
type-preserving simplicial map from $X$ to $\cA$ with the property that for every apartment $\cA'$
containing $a$ the restriction of $\rho_{\cA,a}$ to $\cA'$ is the
unique type-preserving isomorphism of $\cA'$ onto $\cA$ fixing~$a$.
From this one readily deduces the following facts:
\begin{enumerate}
\item[(R1)] For all $c\in\cC$ we have $\delta(a,c) = \delta(a, \rho_{\cA,a}(c))$.
\item[(R2)] If $\gamma = (c_0, \ldots, c_n)$ is a gallery of type $(s_1,
\ldots, s_n)$ in $X$, then  the image 
$$\rho_{\cA,a}(\gamma) = (
\rho_{\cA,a}(c_0),s_1,\ldots,s_n,\rho_{\cA,a}(c_n))$$ is a pregallery
of the same type in $\cA$.
\end{enumerate}

We now turn to projections. All of the properties relevant for us can
be found in \cite[Section~4.9]{ABr:08}. For any simplex $A \in X$
and any chamber $c \in \cC$, there is a unique chamber in $\cC(A)$,
called the \textit{projection of $c$ onto $A$} and denoted by
$\proj_A(c)$, such that $\dist(\proj_A(c),c) < \dist(d,c)$ for all
$d \in \cC(A) \backslash \{\proj_A(c)\}$. For any two simplices $A,B
\in X$, there is a unique simplex containing $A$, called the
\textit{projection $\proj_A(B)$ of $B$ onto $A$}, such that
$\{\proj_A(c) \mid c \in \cC(B)\} = \cC(\proj_A(B))$. We shall use
the following properties of projections:
\begin{enumerate}
\item[(P1)] If $A \in X_I$, $B \in X_J$ and $w = \delta(A,B)$, then
$\proj_A(B)$ has cotype $I \cap wJw^{-1}$. Furthermore,
$$\{c \in \cC(A) \mid  \text{ there exists } d \in \cC(B)
\text{ such that } \delta(c,d) = w \} = \cC(\proj_A(B)).$$
\item[(P2)] There exist bijections (induced by projection maps) between
$\cC(\proj_A(B))$ and $\cC(\proj_B(A))$.
\end{enumerate}

\subsection{Generalised spheres in regular buildings}\label{sec:reg}

Spheres (and balls) in buildings can be defined using the gallery distance. In this paper we study the refinement of this notion by considering spheres defined using the Weyl distance. For a building $X$ of type $(W,S)$, $c\in\cC$ and $w\in W$, let
$$
\cC_w(c)=\{d\in\cC\mid\delta(c,d)=w\}.
$$
And for $I,J \subseteq S, A\in X_I$ and $w\in R(I,J)$, we define
$$
F(A,J,w)=\{B\in X_J\mid \delta(A,B)=w\}.
$$
That is what we mean by a \textit{generalised sphere}. Note that for
$c \in \cC$ and $w \in W = R(\emptyset, \emptyset)$ we have
$F(c,\emptyset,w) = \cC_w(c)$. It is one of the purposes of the
present paper to compute the cardinalities $|F(A,J,w)|$ for
regular, locally finite buildings and spherical cotype~$I$ (if $I$ is not spherical then the cardinality $|F(A,J,w)|$ is often infinite).

A building $X$ is called \textit{regular} if there is a bijection
between $\cC(\pi)$ and $\cC(\pi')$ whenever $\pi$ and $\pi'$ are
panels of the same cotype. If $X$ is thick and $m_{st}<\infty$ for
all $s,t\in S$ then $X$ is necessarily regular (see for example,
\cite[Theorem~2.4]{Par:06}). If each panel is contained in only
finitely many chambers then we call $X$ \textit{locally finite}.
For the remainder of this paper we assume that $X$ is a regular, locally finite building of type $(W,S)$.

For each $s\in S$ we set
$$
q_s+1=|\cC(\pi_s)|\qquad\textrm{for any $s$-panel $\pi_s$}.
$$
The \textit{parameters} of $X$ are the integers~$(q_s)_{s\in S}$.
Note that $q_s = |\cC_s(c)|$ for any $c \in \cC$. Induction on
$\ell(w)$ shows that if $w\in W$ has a reduced decomposition
$w=s_1\cdots s_{\ell}$, then
\begin{align}\label{eq:c}
|\cC_w(c)|=q_{s_1}\cdots q_{s_{\ell}},
\end{align}
see \cite[Proposition~2.1]{Par:06}. In particular, the number $q_w=q_{s_1}\cdots q_{s_{\ell}} =
|\cC_w(c)|$ (with the convention that $q_1 = 1$) does not depend on
the choice of the chamber $c \in \cC$ or the reduced expression
for~$w$. Therefore $q_{uv}=q_uq_v$ whenever
$\ell(uv)=\ell(u)+\ell(v)$. It also follows that $q_s=q_t$ whenever
the order of $st$ in $W$ is finite and odd (but we will not
explicitly use this here).

If $A\in X_I$ and $c\in\cC(A)$ then
$$
\cC(A)=\bigsqcup_{w\in W_I}\cC_w(c).
$$
Thus if $I$ is spherical and $A\in X_I$ then the cardinality $N(I):=|\cC(A)|$ is finite and does not depend on $A\in X_I$, since
\begin{align}\label{eq:firstcard}
N(I)=|\cC(A)|=\sum_{w\in W_I}q_w.
\end{align}

If $I,J\subseteq S$ with $I$ spherical, $A \in X_I$, $B \in X_J$ and $w
= \delta(A,B)$ then properties (P1) and (P2) in Section~\ref{sec:retproj} imply that
\begin{align}\label{eq:swap}
N(I \cap wJw^{-1}) = N(J \cap w^{-1}Iw).
\end{align}

\section{Counting simplices at given Weyl distance}\label{sec:2}

In the following theorem we give our explicit formula for the cardinality of generalised spheres in a regular building.

\begin{thm}\label{thm:sphere} Let $I$ be a spherical subset of $S$, $J$ a subset of $S$, and suppose that $A\in X_I$. If $w\in R(I,J)$ then
$$
|F(A,J,w)|=\frac{N(I)}{N(I\cap wJw^{-1})}q_w,
$$
where $N(K)$, for $K$ spherical, is given by (\ref{eq:firstcard}).
\end{thm}

\begin{proof}
Let $c\in\cC(A)$ and $d\in\cC_w(c)$. Then $d_J\in F(A,J,w)$. If $c'\in \cC(A)$ with $c'\neq c$ then $\cC_w(c)\cap \cC_w(c')=\emptyset$, for if $d\in\cC_w(c)$ and $\delta(c',c)=w_I\in W_I$ then $\delta(c',d)=w_Iw$ (since $w\in R(I,J)$). Hence if $d\in\cC_w(c)$ then $d\notin \cC_w(c')$. 

The above observations show that there is a well defined map
\begin{align*}
f:\bigsqcup_{c\in\cC(A)}\cC_w(c)&\to F(A,J,w),\quad\textrm{given by $f(d)=d_J$}.
\end{align*}
By (\ref{eq:c}) and~(\ref{eq:firstcard}) the domain of $f$ has cardinality $N(I)q_w$. We now compute the cardinalities of the fibres of~$f$. For $B\in F(A,J,w)$ we have
\begin{align*}
f^{-1}(B)&=\{d\in\cC(B)\mid\textrm{there exists $c\in\cC(A)$ such that $\delta(c,d)=w$}\},
\end{align*}
and by (P1) in Section~\ref{sec:retproj} it follows that
\begin{align*}
f^{-1}(B)&=\{d\in\cC(B)\mid \proj_B(A)\subseteq d\}=\cC(\proj_B(A)).
\end{align*}
The cotype of the simplex $\proj_B(A)$ is $J\cap w^{-1}Iw$ (see (P1) in Section~\ref{sec:retproj}), and thus by~(\ref{eq:firstcard}) we have $|f^{-1}(B)|=N(J\cap w^{-1}Iw)\geq 1$. This shows firstly that $f$ is surjective, and secondly that
$$
|F(A,J,w)|=\frac{N(I)q_w}{N(J\cap w^{-1}Iw)}.
$$
The claim now follows since $N(J\cap w^{-1}Iw)=N(I\cap wJw^{-1})$ by~(\ref{eq:swap}).
\end{proof}

\begin{remark} 
We make the following remarks. 
\begin{enumerate}
\item[(a)] In the case of \textit{special vertices} of a fixed type in an affine building, our formula in Theorem~\ref{thm:sphere} recovers~\cite[Proposition~2.7]{Car:99} (in the case of $\tilde{A}_n$ buildings) and \cite[Theorem~5.15]{Par:06} (for general affine buildings).
\item[(b)] We outline a representation theoretic proof of Theorem~\ref{thm:sphere} in Remark~\ref{rem:secondproof}. An advantage of the building theoretic proof presented above is that a combinatorial significance is attached to the denominator $N(I\cap wJw^{-1})$ in the formula. 
\item[(c)] We note that for all finite Coxeter groups $W_I$ there is an explicit closed formula for the polynomial $N(I)$ (this polynomial is often called the \textit{Poincar\'{e} polynomial} of $W_I$). See, for example \cite{Mac:72}. Thus the formula in Theorem~\ref{thm:sphere} can be made completely explicit. 
\end{enumerate}
\end{remark}

The \textit{gallery distance} between simplices $A$ and $B$ is $\mathrm{dist}(A,B)=\ell(\delta(A,B))$. The following elementary corollary gives formulae for the cardinalities of spheres with respect to this numerical distance. 

\begin{cor}\label{cor:sphere}
If $I$ is a spherical subset of $S$, $J$ is a subset of $S$, and $A\in X_I$, then
\begin{align}\label{eq:cor1}
|\{B\in X_J\mid \mathrm{dist}(A,B)=n\}=\sum_{\{w\in R(I,J) \mid\ell(w)=n\}}\frac{N(I)}{N(I\cap wJw^{-1})}q_w.
\end{align}
In particular, if $A\in X_I$ is a vertex of spherical cotype~$I=S\backslash\{s\}$, and if $B$ is a vertex of $X$ with $\mathrm{dist}(A,B)=1$, then $B$ is a vertex of cotype~$I$ and
\begin{align}\label{eq:cor2}
|\{B\in X_I\mid \mathrm{dist}(A,B)=1\}|=\frac{N(I)q_s}{N(I\cap sIs)}=\frac{N(I)q_s}{N(\{t\in I\mid st=ts\})}.
\end{align}
\end{cor}

\begin{proof}
Equation~(\ref{eq:cor1}) follows immediately from Theorem~\ref{thm:sphere}. For equation~(\ref{eq:cor2}), note that if $A\in X_I$ is a vertex of cotype $I=S\backslash\{s\}$ and if $B\in X_J$ is a vertex of cotype $J=S\backslash\{s'\}$  with $\mathrm{dist}(A,B)=1$ then $\delta(A,B)=t$ for some $t\in S$ with $t\in R(I,J)$. Thus $t\notin I\cup J$, and hence $I=J$ and $t=s$. Thus $B$ is a vertex of cotype $I$, and if $I$ is spherical then equation~(\ref{eq:cor1}) gives $|\{B\in X_I\mid \dist(A,B)=1\}|=N(I)q_s/N(I\cap sIs)$. It is elementary that $I\cap sIs=\{t\in I\mid st=ts\}$, hence~(\ref{eq:cor2}).
\end{proof}

In the literature~\cite{SS:96,Set:08} computing the cardinality of the sphere of vertices at numerical distance~$1$ from a given vertex has been an issue even for classical affine buildings of types~$\tilde{A}_n$ and $\tilde{C}_n$ over local fields. Corollary~\ref{cor:sphere}, and in particular~(\ref{eq:cor2}), gives a very general and explicit solution to this problem. 

\begin{example} Let $X$ be an affine building of type~$\tilde{A}_n$ with thickness~$q$. By~(\ref{eq:cor2}) the number of vertices at numerical distance~$1$ from a given vertex of cotype $I=S\backslash\{s\}$ is
\begin{align}\label{eq:conj}
\frac{N(I)}{N(I\cap sIs)}q=\frac{|A_n(q)|}{|A_{n-2}(q)|}q=\frac{(q^{n+1}-1)(q^n-1)}{(q-1)^2}q,
\end{align}
where we use the elementary formula for the number of chambers of an $A_n$ building $A_n(q)$ with thickness~$q$ (or, alternatively, the well known closed form for the Poincar\'{e} polynomial). The formula~(\ref{eq:conj}) was conjectured by Schwartz and Shemanske in~\cite{SS:96}, and proven by Setyadi in~\cite{Set:08}.
\end{example}

\begin{example} 
Let $X$ be an affine building of type~$\tilde{C}_n$, and assume that~$q_s=q$ for all $s\in S$. 
\smallskip

\noindent (a) Suppose first that $s\in S$ is a special type (that is, $s$ is an end node of the Coxeter diagram). Then by~(\ref{eq:cor2}) the number of vertices at numerical distance~$1$ from a given cotype~$I=S\backslash\{s\}$ vertex is
$$
\frac{N(I)}{N(I\cap sIs)}q=\frac{|C_n(q)|}{|C_{n-1}(q)|}q=\frac{q^{2n}-1}{q-1}q,
$$
where we use elementary and well known formulae for the number of chambers of a $C_n(q)$ building with thickness~$q$. This recovers~\cite[Proposition~3.9]{Set:08}. 
\medskip

\noindent (b) Suppose now that $s\in S$ is not a special type. If $s$ is at length $k$ from an end of the Coxeter diagram, then $W_I=C_k\times C_{n-k}$ and $W_{I\cap sIs}=C_{k-1}\times C_{n-k-1}$, and hence the number of vertices at numerical distance~$1$ from a given vertex of cotype $I=S\backslash\{s\}$ is
$$
\frac{N(I)}{N(I\cap sIs)}q=\frac{|C_k(q)|\times |C_{n-k}(q)|}{|C_{k-1}(q)|\times|C_{n-k-1}(q)|}q=\frac{(q^{2k}-1)(q^{2(n-k)}-1)}{(q-1)^2}q.
$$
\end{example}

\section{Intersection numbers in regular buildings}\label{sec:3}

Let $\Gamma=(V,E)$ be a locally finite graph, and for $k\in\mathbb{N}$ let $F_k(x)=\{y\in V\mid d(x,y)=k\}$ be the sphere of radius $k$ centred at~$x\in V$. As recalled in the introduction, the graph $\Gamma$ is \textit{distance regular} if for all $k,\ell\in\mathbb{N}$ we have 
$
|F_k(x)\cap F_{\ell}(y)|=|F_k(x')\cap F_{\ell}(y')|$ whenever $d(x,y)=d(x',y')$. In other words, there are numbers $a_{k,\ell}^m\in\mathbb{N}$ such that 
$$
a_{k,\ell}^m=|F_k(x)\cap F_{\ell}(y)|\quad\text{whenever $y\in F_m(x)$}.
$$ 
Distance regular graphs have been extensively studied, and have a rich algebraic theory centring around the \textit{intersection numbers} $a_{k,\ell}^m$ (see for example \cite{BCN:89}). 

In this section we prove a generalised notion of distance regularity for regular buildings using the Weyl distance between simplices. We also develop a combinatorial formula for the corresponding intersection numbers. Consider first the simplest case, where the simplices are chambers. By \cite[Proposition~3.9]{Par:06} there are numbers $c_{u,v}^w\in \mathbb{N}$ ($u,v,w\in W$) such that
\begin{align}\label{eq:chamberreg}
c_{u,v}^w=|\cC_u(a)\cap \cC_v(b)|\quad\textrm{whenever $a,b\in \cC$ with $b\in\cC_w(c)$}.
\end{align}
We now extend this chamber distance regularity to a general distance regularity result for regular buildings.

\begin{thm}\label{thm:reg} Let $I,J,K$ be spherical subsets of $S$, and suppose that $u\in R(I,J)$, $v\in R(K,J)$, and $w\in R(I,K)$. There are numbers $c_{u,v}^w(I,J,K)\in\mathbb{N}$ such that
$$
c_{u,v}^w(I,J,K)=|F(A,J,u)\cap F(B,J,v)|\quad\text{whenever $A\in X_I$ and $B\in F(A,K,w)$}.
$$
Moreover, for any $z\in W_IwW_K$ we have 
\begin{align}\label{eq:firstform}
c_{u,v}^w(I,J,K)=\frac{1}{N(J)}\sum_{\substack{x\in W_IuW_J\\ y\in W_KvW_J}}c_{x,y}^{z}.
\end{align}
where the numbers $c_{x,y}^z\in\mathbb{N}$ are from~(\ref{eq:chamberreg}). 
\end{thm}

\begin{proof}
Let $A\in X_I$ and $B\in F(A,K,w)$, and let $a,b$ be any chambers with $a\in\cC(A)$ and $b\in\cC(B)$ and set $z=\delta(a,b)$. We claim that
$$
\cC(F(A,J,u)\cap F(B,J,v))=\bigsqcup_{\substack{x\in W_IuW_J\\ y\in W_KvW_J}}\cC_x(a)\cap \cC_{y}(b).
$$
For if $d\in\cC(F(A,J,u)\cap F(B,J,v))$ then $D=d_J\in F(A,J,u)\cap F(B,J,v)$ and so $\delta(A,D)=u$ and $\delta(B,D)=v$. Thus $\delta(a,d)\in W_IuW_J$ and $\delta(b,d)\in W_KvW_J$. Conversely, if $d\in\cC_x(a)\cap \cC_y(b)$ with $x\in W_IuW_J$ and $y\in W_KvW_J$ and $D=d_J$ then $\delta(A,D)=u$ and $\delta(B,D)=v$. Thus $D\in F(A,J,u)\cap F(B,J,v)$, and so $d\in \cC(D)\subseteq \cC(F(A,J,u)\cap F(B,J,v))$. 

Hence, by (\ref{eq:chamberreg}), we have 
$$
|\cC(F(A,J,u)\cap F(B,J,v))|=\sum_{\substack{x\in W_IuW_J\\ y\in W_KvW_J}}c_{x,y}^z.
$$
Since each $J$-simplex of $F(A,J,u)\cap F(B,J,v)$ is contained in precisely $N(J)$ chambers of $\cC(F(A,J,u)\cap F(B,J,v))$ formula~(\ref{eq:firstform}) holds. Thus $|\cC(F(A,J,u)\cap F(B,J,v))|$ depends only on $u,v$ and~$w$ (for we could choose $a\in\cC(A)$ and $b\in\cC(B)$ with $\delta(a,b)=w$). 
\end{proof}

\begin{remark}\label{rem:finitesupport} If $I,J,K$ are spherical and if $u\in R(I,J)$, $v\in R(K,J)$, and $w\in R(I,K)$ then $c_{u,v}^w(I,J,K)=0$ if $\ell(w)>\ell(u)+\ell(v)+\ell(w_J)$, where $w_J\in W_J$ is the longest element of~$W_J$. To see this, note that if $A\in X_I$ and $B\in X_K$ with $\delta(A,B)=w$, and if $D\in F(A,J,u)\cap F(B,J,v)$, then there are chambers $a\in\cC(A)$, $d_1,d_2\in\cC(D)$, and $b\in\cC(B)$ with $\delta(a,d_1)=u$ and $\delta(d_2,b)=v^{-1}$. Thus 
$
\ell(w)=\ell(\delta(A,B))\leq \ell(u)+\ell(\delta(d_1,d_2))+\ell(v^{-1})\leq \ell(u)+\ell(v)+\ell(w_J).
$
\end{remark}

The numbers $c_{u,v}^w(I,J,K)$ are called \textit{intersection numbers}. The following refined formula for the intersection numbers will be used for our combinatorial formula in Theorem~\ref{thm:pathsv1}. It appears that the formula in Proposition~\ref{prop:refinement} does not easily follow from the formula in Theorem~\ref{thm:reg} by obvious algebraic manipulations, despite the apparent similarity in the formulae. Thus we give a direct and independent proof below.

\begin{prop}\label{prop:refinement}
Let $I,J,K\subseteq S$ be spherical and suppose that $u\in R(I,J)$, $v\in R(K,J)$, and $w\in R(I,K)$. For any $z\in W_IwW_K$ we have (with $M_K(J,v)$ as in Section~\ref{sec:doublecosets}):
$$
c_{u,v}^w(I,J,K)=\sum_{\substack{x\in W_IuW_J\\ y\in M_K(J,v)v}}c_{x,y}^z.
$$
\end{prop}

\begin{proof} Choose $A\in X_I$ and $B\in F(A,K,w)$, and fix chambers $a\in\cC(A)$ and $b\in\cC(B)$ such that $\delta(a,b)=z$. We claim that the mapping $c\mapsto c_J$ gives a bijection
$$
\bigsqcup_{\substack{x\in W_IuW_J\\ y\in M_K(J,v)v}}\cC_x(a)\cap \cC_y(b)\to F(A,J,u)\cap F(B,J,v).
$$
To check surjectivity, let $D\in F(A,J,u)\cap F(B,J,v)$. Since $b\in\cC(B)$ and $\delta(\cC(B)\times\cC(D))=W_KvW_J=M_K(v,J)vW_J$ we have $\delta(\{b\}\times \cC(D))= mvW_J$ for some $m\in M_K(J,v)$. Thus there is a chamber $d\in\cC(D)$ with $\delta(b,d)=mv\in M_K(J,v)v$. Since $d\in\cC(D)$ and $D$ has cotype~$J$, we have $d_J=D$. Since $a\in\cC(A)$, $d\in\cC(D)$ and $\delta(A,D)=u$ we have $\delta(a,d)\in W_IuW_J$. 

To check injectivity, suppose that $d,d'\in\cC$ with $\delta(b,d)=mv\in M_K(J,v)v$ and $\delta(b,d')=m'v\in M_K(J,v)v$ and $d_J=d_J'$. Then $\delta(d,d')\in W_J$, and so $m'v=\delta(b,d')\in\delta(b,d)W_J$. Thus $m'v=mvw_J$ for some $w_J\in W_J$, and it follows from Lemma~\ref{lem:exp} that $m'=m$ and $w_J=1$. Thus $d=d'$. The formula now follows by taking cardinalities and using~(\ref{eq:chamberreg}). 
\end{proof}

We now develop a combinatorial formula for the intersection cardinalities $c_{u,v}^w(I,J,K)$ in terms of ``pointed pregalleries'' (see the definition in the following paragraph). These combinatorial objects are inspired by the alcove walk model developed by Ram~\cite{Ram:06}. For the readers familiar with this theory, our pointed pregalleries are analogues of positively folded alcove walks where instead of the `folding' occurring away from a point at infinity, instead the folding occurs away from a fixed chamber. Note that our pointed pregalleries are defined for all Coxeter systems, whereas positively folded alcove walks are only defined for affine Coxeter systems.

Let $\cA$ be an apartment of $X$ and let $a$ be a chamber of~$\cA$. An \textit{$a$-pointed pregallery} in $\cA$ of type $(s_1,\ldots,s_n)$ is a pregallery $\gamma=(c_0,s_1,c_1,s_2,\ldots,s_n,c_n)$ in $\cA$ such that 
$$
\text{if}\quad c_{j-1}=c_j\quad\text{then}\quad \delta(a,c_{j-1})s_j<\delta(a,c_{j-1}).$$ 
Our interest in $a$-pointed pregalleries stems from the following elementary fact. 

\begin{lemma}\label{lem:motivation}
Let $a,b\in\cC$ and let $\cA$ be an apartment with $a,b\in\cA$. If $\Gamma$ is a gallery in $X$ of type $\boldsymbol{v}$ starting at $b$ then $\rho_{\cA,a}(\Gamma)$ is an $a$-pointed pregallery in $\cA$ of type $\boldsymbol{v}$ starting at~$b$
\end{lemma}

\begin{proof}
Write $\rho=\rho_{\cA,a}$. In view of (R2) in Section~\ref{sec:retproj} it suffices to show that $\rho(\Gamma)$ is $a$-pointed. Write $\Gamma=(b_0,s_1,\ldots,s_n,b_n)$ and suppose that $\rho(b_{j-1})=\rho(b_j)$. If $\delta(a,\rho(b_{j-1}))s_j>\delta(a,\rho(b_{j-1}))$ then $\delta(a,b_{j-1})s_j>\delta(a,b_{j-1})$ by (R1), and hence $\delta(a,b_j)=\delta(a,b_{j-1})s_j$. Thus by (R1) we have $\delta(a,\rho(b_j))=\delta(a,\rho(b_{j-1}))s_j$, contradicting $\rho(b_{j-1})=\rho(b_j)$. Thus $\delta(a,\rho(b_{j-1}))s_j<\delta(a,\rho(b_{j-1}))$ and so $\rho(\Gamma)$ is $a$-pointed.
\end{proof}

If $\gamma=(c_0,s_1,\ldots,s_n,c_n)$ is an $a$-pointed pregallery in an apartment~$\cA$, and $s\in S$ then we define
\begin{align*}
\alpha_s(\gamma)&=|\{j\mid s_j=s\textrm{ and }\delta(a,c_{j-1})<\delta(a,c_{j})\}|\\
\sigma_s(\gamma)&=|\{j\mid s_j=s\textrm{ and }c_{j}=c_{j-1}\}|.
\end{align*}
(Intuitively, $\alpha_s(\gamma)$ counts the number of ``ascent steps'' in the pointed pregallery that occur on cotype~$s$ panels, where ascent step means that the gallery increases length from~$a$, and $\sigma_s(\gamma)$ counts the number of stutters in the gallery that occur on cotype~$s$ panels). 

If $\gamma$ is an $a$-pointed pregallery we define
$$
q(\gamma)=\prod_{s\in S}q_s^{\alpha_s(\gamma)}(q_s-1)^{\sigma_s(\gamma)}.
$$

\begin{prop}\label{prop:inverse} Let $X$ be a thick regular building of type $W$ with parameters $(q_s)_{s\in S}$. Let $a,b\in\cC$ with $\delta(a,b)=w$ and let $\cA$ be an apartment containing $a$ and $b$. Let $\boldsymbol{v}=(s_1,\ldots,s_n)\in S^n$ and define
\begin{align*}
\cG(b,\boldsymbol{v})&=\{\textrm{galleries in $X$ of type $\boldsymbol{v}$ starting at $b$}\},\\
\cP_{\cA,a}(b,\boldsymbol{v})&=\{\textrm{$a$-pointed pregalleries in $\cA$ of type $\boldsymbol{v}$ starting at $b$}\}.
\end{align*}
Then 
\begin{enumerate}
\item $\rho_{\cA,a}(\cG(b,\boldsymbol{v}))=\cP_{\cA,a}(b,\boldsymbol{v})$, and
\item $|\rho_{\cA,a}^{-1}(\gamma)|=q(\gamma)$ for all $\gamma\in \cP_{\cA,a}(b,\boldsymbol{v})$ 
\end{enumerate}
\end{prop}

\begin{proof} 
Write $\rho=\rho_{\cA,a}$. From Lemma~\ref{lem:motivation} we have $\rho(\cG(b,\boldsymbol{v}))\subseteq\cP_{\cA,a}(b,\boldsymbol{v})$. We prove the remaining statements by induction on~$n$. Thus suppose that the result is true for $\boldsymbol{v}'=(s_1,\ldots,s_n)$ and let $\boldsymbol{v}=(s_1,\ldots,s_n,t)$. Let
$$
\gamma=(c_0,s_1,c_1,\ldots,s_n,c_n,t,c_{n+1})\in\cP_{\cA,a}(b,\boldsymbol{v})
$$
(and so $c_0=b$) and let $\gamma'=(c_0,s_1,c_1,\ldots,s_n,c_n)\in \cP_{\cA,a}(b,\boldsymbol{v}')$. By the induction hypothesis there are $q(\gamma')$ galleries $\Gamma'=(b_0,s_1,\ldots,s_n,b_n)\in\cG(b,\boldsymbol{v}')$ with $\rho(\Gamma')=\gamma'$ (that is, $\rho(b_j)=c_j$ for all $0\leq j\leq n$). For each such gallery $\Gamma'$ the number of ways of extending $\Gamma'$ to a gallery $\Gamma=(b_0,s_1,\ldots,s_n,b_n,t,b_{n+1})$ with $\rho(\Gamma)=\gamma$ is equal to the number of chambers $b_{n+1}$ such that $\delta(b_n,b_{n+1})=t$ and $\rho(b_{n+1})=c_{n+1}$. 

\noindent\textbf{Case 1:} Suppose that $\delta(a,c_n)t>\delta(a,c_n)$. Thus $c_{n+1}\neq c_n$ (because $\gamma$ is $a$-pointed) and hence $\delta(a,c_{n+1})=\delta(a,c_n)t$. Since $c_n=\rho(b_n)$ we have $\delta(a,c_n)=\delta(a,b_n)$, and so $\delta(a,b_n)t>\delta(a,b_n)$. This implies that $\delta(a,b_{n+1})=\delta(a,b_n)t$ for all $b_{n+1}$ with $\delta(b_n,b_{n+1})=t$, and therefore $\delta(a,\rho(b_{n+1}))=\delta(a,c_n)t=\delta(a,c_{n+1})$ and so $\rho(b_{n+1})=c_{n+1}$. Thus there are $q_t$ extensions $\Gamma$ of $\Gamma'$ so that $\rho(\Gamma)=\gamma$, and so 
$$
|\rho^{-1}(\gamma)|=|\rho^{-1}(\gamma')|q_t=q(\gamma')q_t.
$$
We have $\alpha_t(\gamma)=\alpha_t(\gamma')+1$, $\alpha_s(\gamma)=\alpha_s(\gamma')$ for all $s\neq t$, and $\sigma_s(\gamma)=\sigma_s(\gamma')$ for all $s\in S$, and so $q(\gamma')q_t=q(\gamma)$ completing the induction in this case.

\noindent\textbf{Case 2:} Suppose that $\delta(a,c_n)t<\delta(a,c_n)$ and that $c_n\neq c_{n+1}$, and so $\delta(a,c_{n+1})=\delta(a,c_n)t$. Since $\rho(b_n)=c_n$ we have $\delta(a,b_n)t<\delta(a,b_n)$. Therefore $\rho(b_{n+1})=c_{n+1}$ if and only if $\delta(a,b_{n+1})=\delta(a,c_n)t=\delta(a,b_n)t$, and this occurs for $b_{n+1}\in \cC_t(b_n)$ if and only if $b_{n+1}=\proj_{\pi}(a)$ where $\pi$ is the cotype~$t$ panel of~$b_n$. Thus there is a unique extension $\Gamma$ of $\Gamma'$ so that $\rho(\Gamma)=\gamma$, and so 
$$
|\rho^{-1}(\gamma)|=|\rho^{-1}(\gamma')|=q(\gamma').
$$
Since $\alpha_s(\gamma)=\alpha_s(\gamma')$ and $\sigma_s(\gamma)=\sigma_s(\gamma')$ for all $s\in S$ we have $q(\gamma')=q(\gamma)$. 

\noindent\textbf{Case 3:} Suppose that $\delta(a,c_n)t<\delta(a,c_n)$ and that $c_n= c_{n+1}$. In this case $\rho(b_{n+1})=c_{n+1}$ if and only if $\delta(a,b_{n+1})=\delta(a,c_{n+1})=\delta(a,c_n)=\delta(a,b_n)$, and this occurs for $b_{n+1}\in\cC_t(b_n)$ if and only if $b_{n+1}\neq \proj_{\pi}(b_n)$. Thus there are precisely $q_t-1$ extensions $\Gamma$ of $\Gamma'$ so that $\rho(\Gamma)=\gamma$, and so 
$$
|\rho^{-1}(\gamma)|=|\rho^{-1}(\gamma')|(q_t-1)=q(\gamma')(q_t-1).
$$
We have $\alpha_s(\gamma)=\alpha_s(\gamma')$ for all $s\in S$, $\sigma_s(\gamma)=\sigma_s(\gamma')$ for all $s\neq t$, and $\sigma_t(\gamma)=\sigma_t(\gamma')+1$, and hence $q(\gamma)=q(\gamma')(q_t-1)$ completing the proof.
\end{proof}

A $1$-pointed pregallery in the Coxeter complex~$\Sigma(W,S)$ is simply called a \textit{pointed pregallery}. Thus a pointed pregallery of type $(s_1,\ldots,s_n)$ is a sequence
$$
\gamma=(w_0,s_1,w_1,s_2,\ldots,s_n,w_n)
$$
such that 
\begin{enumerate}
\item $w_0,\ldots,w_n\in W$ with $w_j\in\{w_{j-1},w_{j-1}s_j\}$ for each $j=1,\ldots,n$, and 
\item if $w_{j-1}=w_j$ then $\ell(w_{j-1}s_j)<\ell(w_{j-1})$.
\end{enumerate} 
If $w\in W$ and $\boldsymbol{v}=(s_1,\ldots,s_n)$ let $\cP(w,\boldsymbol{v})$ be the set of all pointed pregalleries in $\Sigma(W,S)$ of type~$\boldsymbol{v}$ starting at~$w$, and for each $u\in W$ let
$$
\cP(w,\boldsymbol{v})_u=\{\gamma\in\cP(w,\boldsymbol{v})\mid \e(\gamma)=u\}.
$$

\begin{thm}\label{thm:formula1}
Let $u,v,w\in W$ and let $\boldsymbol{v}=(s_1,s_2,\ldots,s_n)$ be a reduced expression for $v$. Then
$$
c_{u,v}^w=\sum_{\gamma\in\cP(w,\boldsymbol{v})_u}q(\gamma).
$$
\end{thm}

\begin{proof}
Let $a,b\in\cC(X)$ be chambers of $X$ with $\delta(a,b)=w$. Let $\cA$ be an apartment of $X$ containing~$a$ and~$b$, and let $\rho=\rho_{\cA,a}$. Since $\boldsymbol{v}$ is reduced the map $\cG(b,\boldsymbol{v})\to \cC_v(b)$, $\Gamma\mapsto\e(\Gamma)$ is a bijection. Thus if $c$ is the unique chamber of $\cA$ with $\delta(a,c)=u$ we have
\begin{align*}
|\cC_u(a)\cap \cC_v(b)|&=|\{d\in\cC_v(b)\mid \rho(d)=c\}|=|\{\Gamma\in\cG(b,\boldsymbol{v})\mid \rho(\e(\Gamma))=c\}|.
\end{align*}
Thus by Proposition~\ref{prop:inverse} we have
\begin{align*}
|\cC_u(a)\cap \cC_v(b)|&=|\{\Gamma\in\rho^{-1}(\gamma)\mid \gamma\in \cP_{\cA,a}(b,\boldsymbol{v}),\,\e(\gamma)=c\}|=\sum_{\gamma\in\cP_{\cA,a}(b,\boldsymbol{v})_c}q(\gamma)
\end{align*}
where $\cP_{\cA,a}(b,\boldsymbol{v})_c=\{\gamma\in\cP_{\cA,a}(b,\boldsymbol{v})\mid\e(\gamma)=c\}$. Since $\cA\cong\Sigma(W,S)$ and $\delta(a,b)=w$ we may identify $\cA$ with $\Sigma(W,S)$ such that $a\leftrightarrow 1$ and $b\leftrightarrow w$. Under this identification, $\cP_{\cA,a}(b,\boldsymbol{v})_c$ becomes $\cP(w,\boldsymbol{v})_u$, hence the result. 
\end{proof}

In the following theorem (the main theorem of this section) we present our combinatorial formula for the intersection numbers $c_{u,v}^w(I,J,K)$.

\begin{thm}\label{thm:pathsv1} Let $I,J,K$ be spherical subsets of $S$, and suppose that $u\in R(I,J)$, $v\in R(K,J)$, and $w\in R(I,K)$. Then
$$
c_{u,v}^w(I,J,K)=\sum_{\gamma}q(\gamma)
$$
where the sum is over
$$
\gamma\in\left\{\begin{matrix}\textrm{pointed pregalleries $\gamma$ of type $\boldsymbol{m}\boldsymbol{v}$ in $\Sigma(W,S)$}\\
\textrm{starting at $w$ and with $\e(\gamma)\in W_IuW_J$}\end{matrix}\,\,\,\bigg|\,\,\,m\in M_K(J,v)\right\},
$$
where $\boldsymbol{v}$ is a fixed reduced expression for $v$, and for each $m\in M_K(J,v)$ we fix a reduced expression~$\boldsymbol{m}$. 
\end{thm}

\begin{proof} This follows from Proposition~\ref{prop:refinement} and Theorem~\ref{thm:formula1}.
\end{proof}

\begin{example}\label{ex:F41}
Let $(W,S)$ be an $F_4$ Coxeter system with $S=\{1,2,3,4\}$ in standard Bourbaki labelling and fix $I=\{2,3,4\}$. Consider an $F_4$ building with parameters $q_1=q_2=s$ and $q_3=q_4=t$. The minimal length double coset representatives for $W_I\backslash W/W_I$ are
$$
w_0=e,\quad w_1=1,\quad w_2=12321,\quad w_3=12324321,\quad w_4=123423121324321
$$
where $e$ is the identity element of~$W$. In this example we apply Theorem~\ref{thm:pathsv1} to compute the intersection cardinalities $c_{w_i,w_1}^{w_3}(I,I,I)$ for each $i=0,1,2,3,4$. We have $W_I(I,w_1)=W_{\{3,4\}}$, and $M_I(I,w_1)$ consists of the~$8$ elements
$$
m_0=e,\, m_1=2,\, m_2=32,\, m_3=432,\, m_4=232,\, m_5=2432,\, m_6=32432,\, m_7=232432.
$$
To apply Theorem~\ref{thm:pathsv1} we must find all pointed pregalleries of type $m_kw_1$ (with $0\leq k\leq 7$) starting at $w_3$, and compute $q(\gamma)$ and $\mathrm{end}(\gamma)$ for each such pointed pregallery.

To abbreviate notation, we write a pointed pregallery $\gamma$ of type $i_1\cdots i_n$ starting at $w$ as $\gamma=w\cdot \vec{i}_1\vec{i}_2\cev{i}_3\widehat{i}_4\cdots\cev{i}_n$ (for example), where the symbol $\vec{i}$ indicates that the length increases in that step of the pregallery, $\cev{i}$ indicates length decreases in that step, and $\widehat{i}$ indicates an $i$-stutter in the pregallery. Then $\mathrm{end}(\gamma)$ is simply the word obtained by deleting all terms~$\widehat{i}$.

Note that $\ell(w_jm_k)=\ell(w_j)+\ell(m_k)$ for all $j,k$, and so the first $\ell(m_k)$ steps of a pointed pregallery of type $m_kw_1$ starting at $w_j$ are necessarily $\vec{i}$ steps. Since $\ell(w_31)<\ell(w_3)$ we have two pointed pregalleries of type $m_0w_1$ starting at $w_3$, namely $\gamma_0=w_3\cdot \cev{1}$ and $\gamma_0'=w_3\cdot\widehat{1}$, and $q(\gamma_0)=1$ and $q(\gamma_0')=s-1$. Similarly, for each $k=1,2,3$ we have $\ell(w_3m_k1)<\ell(w_3m_k)$, and so for each of these values of $k$ there are two pointed pregalleries $\gamma_k=w_3\cdot\vec{m}_k\cev{1}$ and $\gamma_k'=w_3\cdot\vec{m}_k\widehat{1}$. We have $q(\gamma_1)=s$, $q(\gamma_2)=st$, $q(\gamma_3)=st^2$, $q(\gamma_1')=(s-1)s$, $q(\gamma_2')=(s-1)st$, and $q(\gamma_3')=(s-1)st^2$.  

On the other hand, for $k=4,5,6,7$ we have $\ell(w_3m_k1)>\ell(w_3m_k)$, and so for these values of $k$ there is only 1 pointed pregallery of type $m_kw_1$ starting at $w_3$, given by $\gamma_k=w_3\cdot\vec{m}_k\vec{1}$. We have $q(\gamma_4)=s^3t$, $q(\gamma_5)=s^3t^2$, $q(\gamma_6)=s^3t^3$, and $q(\gamma_7)=s^4t^3$. 

We see that $\mathrm{end}(\gamma_k')\in W_Iw_3W_I$ for all $k=0,1,2,3$, and $\mathrm{end}(\gamma_0)\in W_Iw_1W_I$, and $\mathrm{end}(\gamma_k)\in W_Iw_2W_I$ for $k=1,2,3$, and $\mathrm{end}(\gamma_k)\in W_Iw_3W_I$ for $k=4,5,6$, and $\mathrm{end}(\gamma_7)\in W_Iw_4W_I$. For example, consider $\gamma_2$. Performing Coxeter moves gives
$\mathrm{end}(\gamma_2)=w_3m_21=12324321321=4(12321)432\in W_Iw_2W_I$.

Thus Theorem~\ref{thm:pathsv1} gives $c_{w_0,w_1}^{w_3}(I,I,I)=0$, $c_{w_1,w_1}^{w_3}(I,I,I)=1$, and
\begin{align*}
c_{w_1,w_1}^{w_3}(I,I,I)&=1,&c_{w_2,w_1}^{w_3}(I,I,I)&=s(t^2+t+1),\\
c_{w_3,w_1}^{w_3}(I,I,I)&=(s-1)(st^2+st+s+1)+s^3t(t^2+t+1),&c_{w_4,w_1}^{w_3}(I,I,I)&=s^4t^3.
\end{align*}
\end{example}

\begin{example}\label{ex:E82}
The formula in Theorem~\ref{thm:pathsv1} is easily implemented in the MAGMA computational algebra system~\cite{MAGMA}, making use of the existing Coxeter group package. For example, consider an $E_8$ Coxeter system with $S=\{1,\ldots,8\}$ in standard Bourbaki labelling and fix $I=S\backslash\{2\}$. Consider an $E_8$ building with parameter~$q$. There are $35$ distinct $W_I\backslash W/W_I$ double cosets. Let $w_i$, with $0\leq i\leq 34$, denote the minimal length representatives of these double cosets. Fixing an order, we take $w_0=e$, $w_1=2$, $w_2=243542$, $w_3=24315436542$, $w_4=243542654376542$, $w_5=2431542654376542$, and $w_6=2435426543176543876542$. Write $C_{i,j}^k=c_{w_i,w_j}^{w_k}(I,I,I)$. Implementing Theorem~\ref{thm:pathsv1} into MAGMA we obtain
\begin{align*}
C_{2,1}^{1}&=\phi_2(q^2)\phi_3(q)\phi_5(q)q^5&
C_{2,1}^{2}&=\phi_2(q^2)(\phi_3(q)^2q^3-1)&
C_{2,1}^{3}&=\phi_3(q)^2q\\
C_{2,1}^{4}&=\phi_3(q^2)\phi_5(q)&
C_{2,1}^{5}&=\phi_3(q)&
C_{2,1}^{6}&=1,
\end{align*}
where $\phi_n(x)=x^{n-1}+\cdots+x+1$ and $C_{2,1}^{k}=0$ for $k=0$ and $7\leq k\leq 34$.
\end{example}

\section{Hecke operators on regular buildings}\label{sec:4}

If $\Gamma$ is a distance regular graph with intersection numbers $a_{k,\ell}^m$ then, as discussed in the introduction, the $k$-adjacency operators
$$
A_kf(x)=\sum_{y\in F_k(x)}f(y)\quad\textrm{for $f:V\to\mathbb{C}$}
$$
satisfy linear relations
$$
A_kA_{\ell}=\sum_{m\in\mathbb{N}}a_{k,\ell}^mA_m\quad\textrm{for all $k,\ell\in \mathbb{N}$}.
$$
Thus the vector space over $\mathbb{C}$ with basis $\{A_k\mid k\in\mathbb{N}\}$ is an algebra~$\scA$.

This brings us to an investigation of generalised adjacency operators on buildings, one of our primary motivations for studying intersection cardinalities of generalised spheres in regular buildings. For each pair $I,J$ of spherical subsets, and each $w\in R(I,J)$, define the \textit{$(I,J,w)$-adjacency operator} $T_w^{IJ}:\{f:X_J\to\CC\}\to\{f:X_I\to\CC\}$ by
$$
(T_w^{IJ}f)(A)=\sum_{B\in F(A,J,w)}f(B),\qquad\textrm{for all $A\in X_I$ and $f:X_J\to\CC$}.
$$
We also refer to these operators as \textit{Hecke operators} on the building.

\begin{prop}\label{prop:linear} Let $I,J$ and $K$ be spherical, and suppose that $u\in R(I,J)$ and $v\in R(K,J)$. Then
$$
T_u^{IJ}T_v^{JK}=\sum_{w\in R(I,K)} c_{u,v^{-1}}^{w}(I,J,K)T_w^{IK},
$$
where for fixed $u$ and $v$, $c_{u,v^{-1}}^w(I,J,K)$ is nonzero for only finitely many $w\in R(I,K)$. 
\end{prop}

\begin{proof} Directly from definition of the adjacency operators we compute, for $A\in X_I$,
\begin{align*}
(T_u^{IJ}T_v^{JK}f)(A)&=\sum_{B\in F(A,J,u)}\sum_{D\in F(B,K,v)}f(D)\\
&=\sum_{D\in X_K}\sum_{B\in X_J}1_{F(A,J,u)}(B)1_{F(B,K,v)}(D)f(D),
\end{align*}
where for any subset $F\subseteq X$ we write $1_{F}$ for the indicator function of~$F$. Since $D\in F(B,K,v)$ if and only if $B\in F(D,J,v^{-1})$ we have $1_{F(B,K,v)}(D)=1_{F(D,J,v^{-1})}(B)$, and thus
\begin{align*}
(T_u^{IJ}T_v^{JK}f)(A)&=\sum_{D\in X_K}|F(A,J,u)\cap F(D,J,v^{-1})|f(D).
\end{align*}
Since $X_K$ is the disjoint union of the sets $F(A,K,w)$ with $w\in R(I,K)$ we have
\begin{align*}
(T_u^{IJ}T_v^{JK}f)(A)&=\sum_{w\in R(I,K)}\sum_{D\in F(A,K,w)}|F(A,J,u)\cap F(D,J,v^{-1})|f(D),
\end{align*}
and the result follows from Theorem~\ref{thm:reg} and Remark~\ref{rem:finitesupport}. 
\end{proof}

\begin{cor}\label{cor:mods} Let $I$ and $J$ be spherical, and let $\scA(I,J)$ be the vector space over $\CC$ spanned by the operators $\{T_w^{IJ}\mid w\in R(I,J)\}$. Then $\{T_w^{IJ}\mid w\in R(I,J)\}$ is a basis of $\scA(I,J)$. Moreover, $\scA(I,I)$ is an algebra, and $\scA(I,J)$ is a left $\scA(I,I)$-module and a right $\scA(J,J)$-module. 
\end{cor}

\begin{proof}
Let $T=\sum_{w\in R(I,J)}a_wT_w^{IJ}$, with $a_w\in\mathbb{C}$. Let $A\in X_I$, and $B\in F(A,J,v)$ with $v\in R(I,J)$, and let $\delta_{B}:X_J\to\mathbb{C}$ be the function $\delta_{B}(B)=1$ and $\delta_B(B')=0$ for all $B'\in X_J\backslash\{B\}$. Then $T\delta_B(A)=a_v$, and hence the operators $\{T_w^{IJ}\mid w\in R(I,J)\}$ are linearly independent, and thus form a basis of $\scA(I,J)$. Proposition~\ref{prop:linear} shows that $\scA(I,I)$ is closed under multiplication, and that $\scA(I,J)$ is closed under left multiplication by elements of $\scA(I,I)$ and right multiplication by elements of $\scA(J,J)$. Hence the result.
\end{proof}

In the language of \textit{association schemes} (see \cite{BCN:89}),   the algebra $\scA(I,I)$ is the \textit{Bose-Mesner algebra} of the natural association scheme on the cotype $I$ simplices of~$X$.

\begin{example} 
Consider the $F_4$ example of Example~\ref{ex:F41}. Write $A_i=T_{w_i}^{II}$ for $i=0,1,2,3,4$. Note that the minimal length double coset representatives $w_i$ are involutions. This implies that~$\scA(I,I)$ is commutative, for if $\delta(A,B)=w\in R(I,I)$ then $\delta(B,A)=w^{-1}=w$ and so
$$
c_{u,v^{-1}}^w(I,I,I)=|F(A,I,u)\cap F(B,I,v^{-1})|=|F(B,I,v)\cap F(A,I,u^{-1})|=c_{v,u^{-1}}^{w}(I,I,I).
$$
By Proposition~\ref{prop:linear}, the computations in Example~\ref{ex:F41} give the coefficient of $A_3$ in the expansions of $A_1A_1$, $A_2A_1$, $A_3A_1$, and $A_4A_1$. We obtain:
\begin{align*}
A_1A_1&=s(s+1)(st+1)(st^2+1)A_0+[s^2(t^2+t+1)+s-1]A_1+(s+1)(st+1)A_2+A_3\\
A_2A_1&=s^3t(t^2+t+1)A_1+(s^2-1)(st+1)A_2+s(t^2+t+1)A_3\\
A_3A_1&=s^4t^3A_1+s^2t^2(s+1)(st+1)A_2+[(s-1)(st^2+st+s+1)+s^3t(t^2+t+1)]A_3\\
&\quad+(s+1)(st+1)(st^2+1)A_4\\
A_4A_1&=s^4t^3A_3+(s^2-1)(st+1)(st^2+1)A_4.
\end{align*}
We note that these formulae are sufficient to compute the character table for the algebra~$\scA(I,I)$. 
\end{example}

\begin{example}
Consider the $E_8$ example of Example~\ref{ex:E82}. Write $A_i=T_{w_i}^{II}$ for $0\leq i\leq 34$. In this case the algebra spanned by the operators $A_i$ is not commutative, for example in the notation of Example~\ref{ex:E82} we have $C_{2,1}^6=1$ and $C_{1,2}^6=0$ (see Corollary~\ref{cor:commute} for more on commutativity). Proposition~\ref{prop:linear} gives
$
A_2A_1=C_{2,1}^1A_1+C_{2,1}^2A_2+C_{2,1}^3A_3+C_{2,1}^4A_4+C_{2,1}^5A_5+C_{2,1}^6A_6.
$
\end{example}

Consider the case $I=J=\emptyset$. Write $T_w^{\emptyset\emptyset}=T_w$ for all $w\in W$ and let $\scH=\scA(\emptyset,\emptyset)$. By Proposition~\ref{prop:linear} we have
\begin{align}\label{eq:baseexpansion}
T_uT_v=\sum_{w\in W}c_{u,v^{-1}}^wT_w\quad\text{for all $u,v\in W$},
\end{align}
where the numbers $c_{u,v^{-1}}^w$ are defined in~(\ref{eq:chamberreg}). From the geometry of the building it is not hard to show that  (see \cite[Theorem~3.4]{Par:06})
\begin{align}\label{eq:pres}
\begin{aligned}
T_wT_s=\begin{cases}T_{ws}&\textrm{if $\ell(ws)>\ell(w)$}\\
q_sT_{ws}+(q_s-1)T_w&\textrm{if $\ell(ws)<\ell(w)$}.
\end{cases}
\end{aligned}
\end{align}
Thus the algebra~$\scH$ is isomorphic to the extensively studied \textit{(Iwahori) Hecke algebra} of the Coxeter system~$(W,S)$ (see~\cite{Lus:03}). Our goal now is to interpret our more general algebras and modules $\scA(I,J)$ in terms of the more familiar Hecke algebra~$\scH$. 

For each spherical subset $I$ of $S$ define an element $\mathbf{1}_I$ of $\scH$ by
$$
\mathbf{1}_I=\frac{1}{N(I)}\sum_{w\in W_I}T_w.
$$
The subalgebra $\mathbf{1}_I\scH\mathbf{1}_I$ of $\scH$ is called a \textit{parabolic Hecke algebra} (see \cite{APV:13,CIK:71}). Note that $\mathbf{1}_I\scH\mathbf{1}_I$ and $\scH$ have different units ($\mathbf{1}_I$ and $T_1$ respectively). 

\begin{lemma}\label{lem:magic} Let $I$ be spherical. Then $T_u\mathbf{1}_I=\mathbf{1}_IT_u=q_u\mathbf{1}_I$ for all $u\in W_I$, and $\mathbf{1}_I^2=\mathbf{1}_I$.
\end{lemma}

\begin{proof} By (\ref{eq:pres}) the vector space spanned by $\{T_u\mid u\in W_I\}$ is a subalgebra of $\scH$. Therefore if $u\in W_I$ we have
$$
T_u\mathbf{1}_I=\frac{1}{N(I)}\sum_{v\in W_I}T_uT_v=\frac{1}{N(I)}\sum_{w\in W_I}\bigg(\sum_{v\in W_I}c_{u,v^{-1}}^w\bigg)T_w.
$$
If $w\in W_I$ and $a,b\in\cC$ with $b\in \cC_w(a)$ then $\sum_{v}c_{u,v^{-1}}^w=\sum_{v}|\cC_u(a)\cap\cC_{v^{-1}}(b)|=|\cC_u(a)|=q_u$. Thus $T_u\mathbf{1}_I=q_u\mathbf{1}_I$, and the proof that $\mathbf{1}_IT_u=q_u\mathbf{1}_I$ is similar. 

Then
$$
\mathbf{1}_I^2=\frac{1}{N(I)}\sum_{w\in W_I}T_w\mathbf{1}_I=\frac{1}{N(I)}\sum_{w\in W_I}q_w\mathbf{1}_I=\frac{N(I)}{N(I)}\mathbf{1}_I=\mathbf{1}_I,
$$
completing the proof. 
\end{proof}

For spherical subsets $I$ and $J$ of $S$ and $w\in R(I,J)$ define an element $P_w^{IJ}\in\mathbf{1}_I\scH\mathbf{1}_J$ by 
$$
P_w^{IJ}=\frac{N(I)}{N(I\cap wJw^{-1})}\mathbf{1}_IT_w\mathbf{1}_J.
$$
The reason for this choice of normalisation will become clear in Theorem~\ref{thm:product}. 
 
\begin{lemma}\label{lem:double} If $I$ and $J$ are spherical subsets of~$S$ and $w\in R(I,J)$ then
$$
P_w^{IJ}=\frac{1}{N(J)}\sum_{z\in W_IwW_J}T_z.
$$
\end{lemma}

\begin{proof} Using~(\ref{eq:exp1})
\begin{align*}
\mathbf{1}_IT_w\mathbf{1}_J&=\frac{1}{N(I)}\sum_{u\in W_I}T_uT_w\mathbf{1}_J=\frac{1}{N(I)}\sum_{x\in M_I(J,w)}\sum_{y\in W_I(J,w)}T_xT_yT_w\mathbf{1}_J.
\end{align*}
Since $w$ is $(I,J)$-reduced, and since $y\in W_I(J,w)$, we have $\ell(yw)=\ell(y)+\ell(w)$. Moreover $yw=wy'$ with $y'=w^{-1}yw\in W_J$ and $\ell(wy')=\ell(w)+\ell(y')$. This implies that $q_{y'}=q_y$ (because $q_yq_w=q_{yw}=q_{wy'}=q_wq_{y'}$), and by (\ref{eq:pres}) and Lemma~\ref{lem:magic} we have
$$
T_yT_w\mathbf{1}_J=T_{yw}\mathbf{1}_J=T_{wy'}\mathbf{1}_J=T_wT_{y'}\mathbf{1}_J=q_{y'}T_w\mathbf{1}_J=q_yT_w\mathbf{1}_J.
$$
Thus by (\ref{eq:stabiliser}) we have $\sum_{y\in W_I(J,w)}T_xT_yT_w\mathbf{1}_J=N(I\cap wJw^{-1})T_xT_w\mathbf{1}_J$, and so using Lemma~\ref{lem:exp},
$$
P_w^{IJ}=\frac{N(I)}{N(I\cap wJw^{-1})}\mathbf{1}_IT_w\mathbf{1}_J=\sum_{x\in M_I(J,w)}T_xT_w\mathbf{1}_J=\frac{1}{N(J)}\sum_{z\in W_IwW_J}T_{z}.\qedhere
$$
\end{proof}

\begin{remark}\label{rem:secondproof} Lemma~\ref{lem:double} gives an independent algebraic  proof of Theorem~\ref{thm:sphere}, as follows. The linear map $\pi:\scH\to\CC$ with $\pi(T_w)= q_w$ for all $w\in W$ satisfies 
$$
\pi(T_w)\pi(T_s)=\begin{cases}
\pi(T_{ws})&\text{if $\ell(ws)>\ell(w)$}\\
q_s\pi(T_{ws})+(q_s-1)\pi(T_w)&\text{if $\ell(ws)<\ell(w)$},
\end{cases}
$$
and so by (\ref{eq:pres}) the map $\pi$ is a $1$-dimensional representation of $\scH$. Applying this representation to the formula in Lemma~\ref{lem:double} gives, for $w\in R(I,J)$,
\begin{align}\label{eq:secondproof}
\frac{N(I)}{N(I\cap wJw^{-1})}q_w=\frac{1}{N(J)}\sum_{z\in W_IwW_J}q_z.
\end{align}
If $A\in X_I$ then $\cC(F(A,J,w))=\bigsqcup_{z\in W_IwW_J}\cC_z(a)$ for any fixed $a\in\cC(A)$. Thus $$|\cC(F(A,J,w))|=\sum_{z\in W_IwW_J}q_z,$$ and since each simplex $B\in F(A,J,w)$ is contained in $N(J)$ chambers of $\cC(F(A,J,w))$, the right hand side of~(\ref{eq:secondproof}) equals $|F(A,J,w)|$, proving Theorem~\ref{thm:sphere}. 
\end{remark}

\begin{cor}\label{cor:basis} If $I$ and $J$ are spherical subsets of~$S$ then the vector space
$
\mathbf{1}_I\scH\mathbf{1}_J$ has basis $\{P_w^{IJ}\mid w\in R(I,J)\}$.
\end{cor}

\begin{proof} Let $z\in W$. Then $z\in W_IwW_J$ for some $w\in R(I,J)$, and we can write $z=xwy$ with $x\in M_I(J,w)$ and $y\in W_J$. Thus $\ell(z)=\ell(x)+\ell(w)+\ell(y)$. Using (\ref{eq:pres}) and Lemma~\ref{lem:magic} we have
$$
\mathbf{1}_IT_z\mathbf{1}_J=\mathbf{1}_IT_xT_wT_y\mathbf{1}_J=q_xq_y\mathbf{1}_IT_w\mathbf{1}_J.
$$
Hence the set $\{P_w^{IJ}\mid w\in R(I,J)\}$ spans $\mathbf{1}_I\scH\mathbf{1}_J$. It follows from Lemma~\ref{lem:double} and the linear independence of the $T_z$, $z\in W$, that the operators $P_w^{IJ}$, $w\in R(I,J)$, are linearly independent. 
\end{proof}

The following theorem, which shows that the multiplication table for the operators $T_w^{IJ}$ and $P_w^{IJ}$ are the same, is the main result of this section. 

\begin{thm}\label{thm:product} Let $I,J,K$ be spherical, and suppose that $u\in R(I,J)$ and $v\in R(K,J)$. Then
$$
P_u^{IJ}P_v^{JK}=\sum_{w\in R(I,K)}c_{u,v^{-1}}^w(I,J,K)P_w^{IK}
$$
where $c_{u,v^{-1}}^w(I,J,K)$ is as in Theorem~\ref{thm:reg}. In particular, the linear map
$$
\begin{matrix}\theta&:&\scA(I,I)&\longrightarrow&\mathbf{1}_I\scH\mathbf{1}_I\\
&&T_w^{II}&\longmapsto&P_{w}^{II}
\end{matrix}\qquad\textrm{is an isomorphism of algebras},
$$
and $\scA(I,J)$ is a left $\mathbf{1}_I\scH\mathbf{1}_I$-module and a right $\mathbf{1}_J\scH\mathbf{1}_J$-module. 
\end{thm}

\begin{proof} By Lemma~\ref{lem:double}, the expansion $T_xT_y=\sum_{z\in W} c_{x,y^{-1}}^zT_z$ (see (\ref{eq:baseexpansion})), and the decomposition of $W$ into $W_IwW_K$ double cosets with $w\in R(I,K)$, we have
\begin{align*}
P_u^{IJ}P_v^{JK}&=\frac{1}{N(J)N(K)}\sum_{\substack{x\in W_IuW_J\\ y\in W_JvW_K}}T_xT_y\\
&=\frac{1}{N(J)N(K)}\sum_{w\in R(I,K)}\sum_{z\in W_IwW_K}\Bigg(\sum_{\substack{x\in W_IuW_J\\
y\in W_JvW_K}}c_{x,y^{-1}}^z\Bigg)T_z.
\end{align*}
Thus by Theorem~\ref{thm:reg} and Lemma~\ref{lem:double} we have
$$
P_u^{IJ}P_v^{JK}=\sum_{w\in R(I,K)}\frac{c_{u,v^{-1}}^w(I,J,K)}{N(K)}\sum_{z\in W_IwW_K}T_z=\sum_{w\in R(I,K)} c_{u,v^{-1}}^w(I,J,K)P_w^{IK}.
$$

It follows that the linear map $\theta:\scA(I,I)\to \mathbf{1}_I\scH\mathbf{1}_I$ with $T_w^{II}\mapsto P_w^{II}$ is an isomorphism of algebras by using the fact that $\scA(I,I)$ has basis $\{T_w^{II}\mid w\in R(I,I)\}$ and $\mathbf{1}_I\scH\mathbf{1}_I$ has basis $\{P_w^{II}\mid w\in R(I,I)\}$. The final statement follows from Corollary~\ref{cor:mods}. 
\end{proof}

\section{Applications}\label{sec:5}

In this final section we describe some applications of our main theorems.

\subsection{Structure constants in parabolic Hecke algebras}

Theorem~\ref{thm:product} gives a combinatorial interpretation of the structure constants in parabolic Hecke algebras, and Theorem~\ref{thm:pathsv1} gives a formula for these structure constants. In the case when $W$ has type $\tilde{A}_n$ and $I=S\backslash\{s_0\}$ these structure constants are \textit{Hall polynomials}, and nonnegativity of these polynomials was proved by Miller Malley~\cite{MM:96}. More generally, if $W$ is of affine type and $I=S\backslash\{s_0\}$ then the structure constants are the structure constants of the algebra spanned by the \textit{Macdonald spherical functions}, and nonnegativity of these structure constants was proved independently by Parkinson~\cite{Par:06} and Schwer~\cite{Sch:06}. Our formula in Theorem~\ref{thm:pathsv1} immediately gives a considerable generalisation of these results to arbitrary type:

\begin{cor}\label{cor:nonneg}
The structure constants in the parabolic Hecke algebra $\mathbf{1}_I\scH\mathbf{1}_I$ relative to the basis $\{P_w^{II}\mid w\in R(I,I)\}$ are nonnegative. 
\end{cor}

\subsection{Commutativity of algebras of Hecke operators on buildings}

Recall that for distance regular graphs the algebra~$\scA$ of averaging operators is necessarily commutative (this boils down to the fact that the graph distance satisfies $d(x,y)=d(y,x)$). In contrast, the algebra $\scA(I,I)$ is not necessarily commutative. Using Theorem~\ref{thm:product} we have the following commutativity classification (where we use standard Bourbaki labelling~\cite{Bou:02}), showing that in fact commutativity is extremely rare. 

\begin{cor}\label{cor:commute}
Suppose that $(W,S)$ is irreducible. Let $I$ be a spherical subset of~$S$. The algebra $\scA(I,I)$ is noncommutative if $|S\backslash I|>1$. If $I=S\backslash\{i\}$ then $\scA(I,I)$ is commutative in the cases
\begin{enumerate}
\item $W=A_n$ and $1\leq i\leq n$, $W=B_n$ and $1\leq i\leq n$, $W=D_n$ and $1\leq i\leq n/2$ or $i=n-1,n$, $W=E_6$ and $i=1,2,6$, $W=E_7$ and $i=1,2,7$, $W=E_8$ and $i=1,8$, $W=F_4$ and $i=1,4$, $W=H_3$ and $i=1,3$, $W=H_4$ and $i=1$, $W=I_2(p)$ and $i=1,2$, or
\item $W$ affine and $i$ is a special type (that is, $i$ is in the orbit of the special node $0$ under the action of diagram automorphisms),
\end{enumerate}
and noncommutative otherwise.
\end{cor}

\begin{proof} By Theorem~\ref{thm:product} the algebra $\scA(I,I)$ is isomorphic to the parabolic Hecke algebra $\mathbf{1}_I\scH\mathbf{1}_I$. The commutative parabolic Hecke algebras are classified in \cite[Theorem~2.1]{APV:13}, and the result follows from this classification.
\end{proof}

\subsection{Combinatorics of double cosets in groups with $BN$-pairs}

 Let $G$ be a group acting on $X$ by type preserving simplicial automorphisms. The group $G$ has a \textit{strongly transitive} action on $X$ (with respect to some fixed system of apartments) if $G$ is transitive on the set of pairs $(\cA,a)$ with $\cA$ an apartment and $a\in\cA$ a chamber of~$\cA$. The group $G$ has a \textit{Weyl-transitive} action on $X$ if for all $w\in W$ the action is transitive on all ordered pairs $(a,b)$ of chambers such that $\delta(a,b)=w$. Recall that strong transitivity implies Weyl transitivity, although the reverse implication does not hold in general (see \cite{ABr:07}). 

If $G$ acts Weyl transitively on~$X$ then it is immediate that $|F(A,J,u)\cap F(A',J,v)|$ depends only on $u,v$ and $\delta(A,A')$, and therefore Theorem~\ref{thm:reg} is most interesting in that case where the building does not admit a Weyl transitive group action. 

We now outline applications of our formula in Theorem~\ref{thm:pathsv1} to double coset combinatorics in groups $G$ acting strongly transitively on $X$. Fix a chamber $a\in \cC$ and an apartment $\cA$ with $a\in \cA$. Let $B=\mathrm{stab}_G(a)$ and $N=\mathrm{stab}_G(\cA)$. The pointwise stabiliser of $\cA$ is $H=B\cap N$, and $N/H$ is isomorphic to $W$ (the Coxeter group of $X$). It is well-known that $(G,B,N,W)$ is a $BN$-pair (also called a Tits system, see \cite[Theorem~6.56]{ABr:08}). The stabiliser of the cotype $I$ simplex $A=a_I$ of $a$ is the \textit{parabolic subgroup}
$
P_I=\bigsqcup_{w\in W_I}BwB
$. 

\begin{prop}\label{prop:groups1}
With the above notation, if $I,J$ and $K$ are spherical, $u\in R(I,J)$, $v\in R(K,J)$, and $w\in R(I,K)$, then 
\begin{align*}
\Card_{G/P_J}(P_IuP_J\cap wP_KvP_J)=c_{u,v}^w(I,J,K).
\end{align*}
Thus Theorem~\ref{thm:pathsv1} gives a formula for these cardinalities.
\end{prop}

\begin{proof} The cotype $I$ simplices $X_I$ of $X$ can be identified with $G/P_I$. Let $A=P_I$ under this identification. If $w\in R(I,J)$ then $F(A,J,w)$ is
the set of $P_J$ cosets in $
P_IwP_J
$.
Therefore if $I,J$ and $K$ are spherical, and if $u\in R(I,J)$, $v\in R(K,J)$, and $w\in R(I,K)$, then 
$$
\Card_{G/P_J}(P_IuP_J\cap wP_KvP_J)=|F(A,J,u)\cap F(wA,J,v)|=c_{u,v}^w(I,J,K).\qedhere
$$
\end{proof}

\subsection{Convolution algebras and Gelfand pairs}

Suppose that $G$ is a locally compact group acting strongly transitively on the building~$X$, and let $B=\mathrm{stab}_G(a)$ be the stabiliser of a fixed chamber (as in the previous section). Assume that $B$ is compact. A function $f:G\to\CC$ is \textit{bi-$B$-invariant} if $f(bgb')=f(g)$ for all $g\in G$ and all $b,b'\in B$. Let $\cF(B\backslash G/B)$ be the space of bi-$B$-invariant complex valued functions supported on finitely many double cosets. Define convolution in $\cF(B\backslash G/B)$ by
$$
(f_1*f_2)(g)=\int_Gf_1(gh)f_2(h^{-1})\,d\mu(h),
$$
where $\mu$ is a Haar measure on $G$ normalised so that $\mu(B)=1$. It is elementary that $\cF(B\backslash G/B)$ is closed under convolution, and so $\cF(B\backslash G/B)$ is an algebra.  

For $I,J\subseteq S$ spherical and $w\in R(I,J)$, let $\psi_{P_IwP_J}=\frac{1}{N(J)}\chi_{P_IwP_J}$, where $\chi_{P_IwP_J}$ is the characteristic function of the double coset $P_IwP_J$.

\begin{prop}\label{prop:groups2}
Let $I,J,K\subseteq S$ be spherical, and let $u\in R(I,J)$ and $v\in R(J,K)$. Then
$$
\psi_{P_IuP_J}*\psi_{P_JvP_K}=\sum_{w\in R(I,K)}c_{u,v^{-1}}^w(I,J,K)\psi_{P_IwP_K}.
$$
\end{prop}

\begin{proof}
The assumption that $I$ and $J$ are spherical ensures that $\psi_{P_IwP_J}$ is supported on finitely many $B$ double cosets, thus $\psi_{P_IwP_J}\in\cF(B\backslash G/B)$. If $g_1\in P_I$ and $g_2\in P_K$ then for all $g\in G$ we have
\begin{align*}
(\psi_{P_IuP_J}*\psi_{P_JvP_K})(g_1gg_2)&=\int_G\psi_{P_IuP_J}(g_1gg_2h)\psi_{P_JvP_K}(h^{-1})\,d\mu(h)\\
&=\int_G\psi_{P_IuP_J}(gh')\psi_{P_JvP_K}(h'^{-1}g_2)\,d\mu(h')=(\psi_{P_IuP_J}*\psi_{P_JvP_K})(g).
\end{align*}
Thus $\psi_{P_IuP_J}*\psi_{P_JvP_K}$ is left $P_I$-invariant and right $P_K$-invariant, and hence is a linear combination of terms $\psi_{P_IwP_K}$ with $w\in R(I,K)$. To compute the coefficient of $\psi_{P_IwP_K}$ in this linear combination, note that from the definition of convolution, and using Proposition~\ref{prop:groups1}, we have
\begin{align*}
\frac{1}{N(J)}(\psi_{P_IuP_J}*\psi_{P_JvP_J})(w)&=\frac{1}{N(J)}\Card_{G/B}(w^{-1}P_IuP_J\cap P_Kv^{-1}P_J)\\
&=\frac{1}{N(J)}\Card_{G/B}(P_IuP_J\cap wP_Kv^{-1}P_J)\\
&=\Card_{G/P_J}(P_IuP_J\cap wP_Kv^{-1}P_J)=c_{u,v^{-1}}^w(I,J,K),
\end{align*}
completing the proof.
\end{proof}

\begin{cor}\label{cor:gelfand1}
With the above notation, let $\cF(P_I\backslash G/P_I)$ be the subalgebra of $\cF(B\backslash G/B)$ consisting of bi-$P_I$-invariant functions. Then
$$
\cF(P_I\backslash G/P_I)\cong \mathbf{1}_I\scH\mathbf{1}_I,
$$
with the isomorphism given by $\psi_{P_IwP_I}\mapsto P_w^{II}$ for $w\in R(I,I)$. 
\end{cor}

\begin{proof} This is immediate from Theorem~\ref{thm:product} and Proposition~\ref{prop:groups2}.
\end{proof}

Recall that for any locally compact group $G$ and compact subgroup $K$, the pair $(G,K)$ is called a \textit{Gelfand pair} if the convolution algebra of compactly supported continuous bi-$K$-invariant functions on $G$ is commutative. 

\begin{cor}\label{cor:gelfand2}
The pair $(G,P_I)$ from Corollary~\ref{cor:gelfand1} is a Gelfand pair if and only if $\scA(I,I)$ is commutative, and hence only in the cases listed in Corollary~\ref{cor:commute}. 
\end{cor}

\begin{proof}
From Theorem~\ref{thm:product} and Corollary~\ref{cor:gelfand1} we have $\scA(I,I)\cong \mathbf{1}_I\scH\mathbf{1}_I\cong \cF(P_I\backslash G/P_I)$ and the result follows.
\end{proof}

Note that Corollary~\ref{cor:gelfand2} implies that the only Gelfand pairs arising from buildings with infinite Coxeter groups are those Gelfand pairs $(G,K)$ where $K$ is the stabiliser of a special vertex in an affine building. These Gelfand pairs have received considerable attention (see for example \cite{CC:15}).

\subsection{Isotropic random walks on the simplices of buildings}
 
 Fix $I\subseteq S$ spherical, and let $(Z_n)_{n\geq 0}$ be a random walk on the set $X_I$ of all cotype $I$ simplices of a locally finite regular building~$X$. Thus $(Z_n)_{n\geq 0}$ is a Markov chain of $X_I$-valued random variables, with evolution governed by \textit{transition probabilities}
$$
p(A,B)=\mathbb{P}[Z_{n+1}=B\mid Z_n=A]\quad\text{for $A,B\in X_I$}.
$$
For simplicity we assume throughout that $(Z_n)_{n\geq 0}$ has \textit{bounded jumps}. That is, for each $A\in X_I$ the probability $p(A,B)$ is nonzero for only finitely many $B\in X_I$. 

One important goal in random walk theory is to obtain estimates for the $n$-step transition probabilities
$$
p^{(n)}(A,B)=\mathbb{P}[Z_n=B\mid Z_0=A],
$$
and in particular for the $n$-step return probabilities $p^{(n)}(A,A)$. Such a result is called a \textit{local limit theorem}, and ideally it takes the form of an asymptotic formula.

A random walk on $X_I$ is called \textit{isotropic} if $p(A,B)=p(A',B')$ whenever $\delta(A,B)=\delta(A',B')$. In other words, a walk is isotropic if the transition probabilities depend only on the Weyl distance. In this section we use Theorem~\ref{thm:product} to reduce the `general' case of isotropic random walks on the cotype~$I$ simplices of a building to the `special' case of isotropic random walks on the set of chambers of the building (where $I=\emptyset$). This latter case admits a rather complete theory when $X$ is an affine building (see \cite{PS:11}), and thus we can now obtain precise local limit theorems for isotropic random walks on the cotype $I$ simplices of an affine building. We will give a concrete example at the end of this section.

The \textit{transition operator} of a random walk~$(Z_n)_{n\geq 0}$ on $X_I$ is the operator $T$ acting on the space of all functions $f:X_I\to \mathbb{C}$ by
$$
Tf(A)=\sum_{B\in X_I}p(A,B)f(B).
$$
Thus the transition operator of an isotropic random walk $(Z_n)_{n\geq 0}$ on $X_I$ is given by
\begin{align}\label{eq:transition}
T=\sum_{w\in R(I,I)}p_wT_w^{II}\quad\text{where $p_w=p(A,B)$ for any $A,B\in X_I$ with $\delta(A,B)=w$},
\end{align}
and so $T$ is an element of the algebra~$\scA(I,I)$ (here we use the bounded jumps assumption to ensure that $T$ is a finite linear combination; in the general case $T$ lies in a completion of the algebra $\scA(I,I)$). It is useful to note that if $\delta(A,B)=w$ then $p^{(n)}(A,B)$ is the coefficient of $T_w^{II}$ in the expansion of $T^n\in\scA(I,I)$ as a linear combination of the basis $\{T_v^{II}\mid v\in R(I,I)\}$.

\begin{prop}\label{prop:probeq}
Let $I\subseteq S$ be spherical. Let $(Z_n)_{n\geq 0}$ be an isotropic random walk on $X_I$ with bounded jumps and transition probabilities $p(A,B)$, and let $p_w=p(A,B)$ for any $A,B\in X_I$ with $\delta(A,B)=w$. Let $(\overline{Z}_n)_{n\geq 0}$ be the isotropic random walk on $\cC$ with transition probabilities
$$
\overline{p}(a,b)=\frac{p_w}{N(I)}\quad\text{if $\delta(a,b)\in W_IwW_I$ with $w\in R(I,I)$}.
$$
Then for all $A,B\in X_I$ and all $a,b\in\cC$ with $\delta(a,b)\in W_I\delta(A,B)W_I$ we have
$$p^{(n)}(A,B)=\frac{1}{N(I)}\overline{p}^{(n)}(a,b)\quad\text{for all $n\in\mathbb{N}$}.
$$ 
\end{prop}

\begin{proof}
We must first check that $\overline{p}(a,b)$ defines a transition probability, and thus we verify that $\sum_{b\in \cC}\overline{p}(a,b)=1$ for each $a\in\cC$. We have
$$
\sum_{b\in\cC}\overline{p}(a,b)=\sum_{w\in R(I,I)}\sum_{\{b\mid \delta(a,b)\in W_IwW_I\}}\overline{p}(a,b)=\sum_{w\in R(I,I)}\frac{p_w}{N(I)}|\{b\mid\delta(a,b)\in W_IwW_I\}|.
$$
The cardinality in the summand is equal to $N(I)|F(A,I,w)|$ where $A$ is the cotype~$I$ simplex of $a$, since each element of $F(A,I,w)$ is contained in exactly $N(I)$ chambers of $\{b\mid\delta(a,b)\in W_IwW_I\}$. Thus
$$
\sum_{b\in\cC}\overline{p}(a,b)=\sum_{w\in R(I,I)}p_w|F(A,I,w)|=\sum_{B\in X_I}p(A,B)=1.
$$ 

Let $T\in\scA(I,I)$ be the transition operator of~$(Z_n)_{n\geq 0}$, as in (\ref{eq:transition}). Let $\theta:\scA(I,I)\to\mathbf{1}_I\scH\mathbf{1}_I$ be the isomorphism from Theorem~\ref{thm:product}. By Lemma~\ref{lem:double}
\begin{align*}
\theta(T)&=\sum_{w\in R(I,I)}p_wP_w^{II}=\sum_{w\in R(I,I)}\frac{p_w}{N(I)}\sum_{z\in W_IwW_I}T_z=\sum_{z\in W}\overline{p}_zT_z,
\end{align*}
where $\overline{p}_z=p_w/N(I)$ if $z\in W_IwW_I$ with $w\in R(I,I)$. Thus $\theta(T)\in\scH$ is the transition operator of the isotropic random walk $(\overline{Z})_{n\geq 0}$ on~$\cC$. 

Let $A,B\in X_I$ with $\delta(A,B)=w\in R(I,I)$. Then $p^{(n)}(A,B)$ is the coefficient of $T_w^{II}$ in the expansion of $T^n$, and this equals the coefficient of $P_w^{II}$ in the expansion of $\theta(T^n)=\theta(T)^n$. By Lemma~\ref{lem:double} this equals $N(I)^{-1}$ times the coefficient of $T_z$ in the expansion of $\theta(T)^n$ for any $z\in W_IwW_I$, and this equals $\overline{p}^{(n)}(a,b)/N(I)$ for any $a,b\in\cC$ with $\delta(a,b)=z\in W_IwW_I$.
\end{proof}

We conclude with a concrete example illustrating how Proposition~\ref{prop:probeq} can be used, in conjunction with the techniques in~\cite{PS:11}, to give new local limit theorems for random walks on simplices of affine buildings. 

Let $X$ be a locally finite thick $\tilde{A}_2$ building with thickness parameter~$q$. Let the vertices of $X$ have types $\{0,1,2\}$ and let $X_0$ be the set of all cotype $0$ simplices of $X$ (that is, the cotype $0$ panels of $X$). The `neighbours' of $A\in X_0$ are those simplices $B\in X_0$ with $d(A,B)=1$, or equivalently, $\delta(A,B)\in\{s_1,s_2\}$. Note that $|\{B\in X_0\mid d(A,B)=1\}|=2q(q+1)$ for each $A\in X_0$. The \textit{simple random walk} on $X_0$ is the random walk with transition probabilities
$$
p(A,B)=\begin{cases}
\frac{1}{2q(q+1)}&\text{if $d(A,B)=1$}\\
0&\text{otherwise}.
\end{cases}
$$
This walk is isotropic, with transition operator
$$
T=\frac{1}{2q(q+1)}(T_{s_1}^{00}+T_{s_2}^{00}).
$$
Let $\overline{p}(a,b)$ be the transition probabilities of the associated isotropic random walk $(\overline{Z}_n)_{n\geq 0}$ on~$\cC$ (as in Proposition~\ref{prop:probeq}). The transition operator of this walk is
$$
\theta(T)=\frac{1}{2q(q+1)^2}\left(T_{s_1}+T_{s_1s_0}+T_{s_0s_1}+T_{s_0s_1s_0}+T_{s_2}+T_{s_2s_0}+T_{s_0s_2}+T_{s_0s_2s_0}\right).
 $$

By Proposition~\ref{prop:probeq} we have $p^{(n)}(A,A)=\overline{p}^{(n)}(a,a)/(q+1)$. An asymptotic formula for $\overline{p}^{(n)}(a,a)$, and hence $p^{(n)}(A,A)$, can be computed using the techniques in~\cite{PS:11}. The details of the calculation are somewhat involved, using the representation theory of the associated affine Hecke algebra. Thus we content ourselves here to simply stating the final result. 

\begin{thm}
For the simple random walk on the set $X_0$ of cotype $0$ simplices of an $\tilde{A}_2$ building we have
$$
p^{(n)}(A,A)\sim \frac{\sqrt{3}(q^2+4q-1)^4}{\pi q(q+1)(q-1)^6}\left(\frac{q^2+4q-1}{2q(q+1)}\right)^nn^{-4}\quad\text{as $n\to\infty$}
$$
for all $A\in X_0$, where $\sim$ denotes asymptotic equivalence.
\end{thm}

\medskip

\noindent\begin{minipage}{0.5\textwidth}
Peter Abramenko\newline
Department of Mathematics\newline
University of Virginia\newline
Charlottesville, VA 22904-4137, USA\newline
\texttt{pa8e@virginia.edu}
\end{minipage}
\begin{minipage}{0.5\textwidth}
\noindent James Parkinson\newline
School of Mathematics and Statistics\newline
University of Sydney\newline
NSW, 2006, Australia\newline
\texttt{jamesp@maths.usyd.edu.au}
\end{minipage}
\bigskip

\noindent\begin{minipage}{0.5\textwidth}
Hendrik Van Maldeghem\newline
Department of Mathematics\newline
Ghent University\newline
Krijgslaan 281, S22,\newline
9000 Ghent, Belgium\newline
\texttt{hvm@cage.UGent.be}
\end{minipage}

\end{document}